\documentclass[12pt,twoside]{amsart}
\usepackage{amsthm,amsfonts,amssymb,amscd,euscript}

\usepackage[matrix,arrow]{xy}

\author{Florin Ambro} 
\address{Institute of Mathematics ``Simion Stoilow'' of the Romanian
Academy\\
P.O. BOX 1-764, RO-014700 Bucharest\\ 
Romania.}
\email{florin.ambro@imar.ro}

\newcommand{\isoto}{{\overset{\sim}{\rightarrow}}}

\newcommand{\C}{{\mathbb C}}
\newcommand{\Q}{{\mathbb Q}}
\newcommand{\Z}{{\mathbb Z}}
\newcommand{\N}{{\mathbb N}}
\newcommand{\R}{{\mathbb R}}

\newcommand{\bA}{{\mathbb A}} 

\newcommand{\cB}{{\mathcal B}}

\newcommand{\cM}{{\mathcal M}}

\newcommand{\cO}{{\mathcal O}}
\newcommand{\cS}{{\mathcal S}}

\newcommand{\Char}{\operatorname{char}}
\newcommand{\codim}{\operatorname{codim}}

\newcommand{\depth}{\operatorname{depth}}

\newcommand{\dv}{\operatorname{div}}

\newcommand{\emb}{\operatorname{emb}}

\newcommand{\Exc}{\operatorname{Exc}}

\newcommand{\Int}{\operatorname{int}}

\newcommand{\LCS}{\operatorname{LCS}}

\newcommand{\lcm}{\operatorname{lcm}}

\newcommand{\mult}{\operatorname{mult}}

\newcommand{\relint}{\operatorname{relint}}
\newcommand{\Res}{\operatorname{Res}}

\newcommand{\Sing}{\operatorname{Sing}}

\newcommand{\Spec}{\operatorname{Spec}}

\newcommand{\Supp}{\operatorname{Supp}}

\theoremstyle{plain}
\newtheorem{thm}{Theorem}[section]

\newtheorem{question}[thm]{Question}
\newtheorem{lem}[thm]{Lemma}
\newtheorem{cor}[thm]{Corollary}

\newtheorem{prop}[thm]{Proposition}

\theoremstyle{definition}
\newtheorem{defn}[thm]{Definition}

\newtheorem{exmp}[thm]{Example}

\newtheorem{rem}[thm]{Remark}

\newtheorem{ack}{Acknowledgments}   

\theoremstyle{remark}

\begin{document}

\bibliographystyle{amsalpha+}
\title{On toric face rings II}
\maketitle

\begin{abstract} We introduce the class of weakly log canonical singularities, a natural generalization of semi-log canonical 
singularities. Toric varieties (associated to toric face rings, possibly non-normal or 
reducible) which have weakly (semi-) log canonical singularities are classified. 
In the toric case, we discuss residues to lc centers of codimension one or higher.
\end{abstract} 





\section*{Introduction}


Our motivation is to better understand semi-log canonical singularities (cf.~\cite{Kbook})
by constructing toric examples. Semi-log canonical singularities are possibly not normal,
and even reducible. So by a toric variety we mean $\Spec k[\cM]$, the spectrum of a 
toric face ring $k[\cM]$ associated to a monoidal complex $\cM=(M,\Delta,(S_\sigma)_{\sigma\in \Delta})$.
From the algebraic point of view, toric face rings were introduced as a generalization of 
Stanley-Reisner rings, studied by Stanley,
Reisner, Bruns, Ichim, R\"omer and others (see the introductions of~\cite{IR07,TFR1} for example).
From the geometric point of view, Alexeev~\cite{Ale02} introduced another generalization of 
Stanley-Reisner rings, the so called stable toric varieties, obtained by glueing toric varieties (possibly 
not affine) along orbits.

In order to understand residues for varieties with normal crossings singularities, we were forced
to enlarge the category of semi-log canonical singularities to the class of weakly log canonical singularities. 
To see this, let us consider the normal crossings model
$\Sigma=\cup_{i=1}^nH_i\subset \bA^n_\C$, where $H_i:(z_i=0)$ is the $i$-th standard hyperplane.
It is Cohen Macaulay and Gorenstein, and codimension one residues onto components of $\Sigma$
glue to a residue isomorphism 
$
\Res\colon \omega_{\bA^n}(\log\Sigma)|_\Sigma \isoto \omega_\Sigma,
$
where $\omega_\Sigma$ is a dualizing sheaf. It follows that 
$\Sigma$ has semi-log canonical singularities and $\omega_\Sigma\simeq \cO_\Sigma$. The
complement $T=\bA^n\setminus \Sigma$ is the $n$-dimensional torus, which acts naturally on $\bA^n$.
The invariant closed irreducible subvarieties of codimension $p$ are $H_{i_1}\cap \cdots \cap H_{i_p}$ for $i_1<\cdots <i_p$.
A natural way to realize $\Sigma$ as a glueing of smooth varieties (cf.~\cite{Del71}) is to consider 
the decreasing filtration of algebraic varieties
$$
X_1\supset X_2\supset \cdots
$$
where $X_1=\Sigma$ and $X_{p+1}=\Sing(X_p)$ for $p\ge 1$. It turns out that $X_p$ is the union of 
$T$-invariant closed irreducible subvarieties of $\bA^n$ of codimension $p$, that is 
$X_p=\cup_{i_1<\cdots <i_p}H_{i_1}\cap \cdots \cap H_{i_p}$ (the reader may check that $X_p$ is the affine toric variety
associated to the following monoidal complex: lattice $\Z^n$, fan consisting of all faces $\sigma\prec \R^n_{\ge 0}$
of codimension at least $p$, and semigroups $S_\sigma=\Z^n\cap \sigma$).
After extending the filtration with $X_0=\bA^n$, we would like to realize it as a chain of semi-log canonical 
structures and (glueing of) codimension one residues
$$
(\bA^n,\Sigma)\leadsto (\Sigma,0)\leadsto (X_2,0) \leadsto\cdots.
$$
The varieties $X_p$ are weakly normal and Cohen Macaulay, but not nodal in codimension one if $p>1$.
The dualizing sheaf of $X_2$ is not invertible in codimension one, so we cannot define the sheaves $\omega^{[n]}_{X_2}\ (n\in \Z)$,
and $(X_2,0)$ is not semi-log canonical. We observe in this paper that
the filtration may still be viewed as a chain of log structures, provided we enlarge the category of semi-log canonical singularities
to a certain class called {\em weakly log canonical singularities}. We show that for $p>0$, $(X_p,0)$ has weakly log
canonical singularities, $\omega^{[2]}_{(X_p,0)}\simeq \cO_{X_p}$, and codimension one residues onto components of $X_{p+1}$
glue to a residue isomorphism 
$
\Res^{[2]} \colon \omega^{[2]}_{(X_p,0)}| _{X_{p+1}} \isoto \omega^{[2]}_{(X_{p+1},0)}
$
(see Proposition~\ref{mex}).

A semi-log canonical singularity $X$ is defined as a singularity such that 
a) $X$ is $S_2$ and nodal in codimension one, b) certain pluricanonical sheaves $\omega^{[r]}_X$ are invertible,
and c) the induced log structure on the normalization has log canonical singularities. We define weakly log canonical singularities
by replacing axiom a) with a'): $X$ is $S_2$ and weakly normal. The known pluricanonical sheaves $\omega^{[r]}_X$ 
are replaced by certain pluricanonical sheaves $\omega^{[r]}_{(X,0)}$, consisting of rational differential $r$-forms on $X$ which
have constant residues over each codimension one non-normal point of $X$.
Semi-log canonical singularities are a subclass of weakly log canonical singularities, as it turns out that 
$\omega^{[r]}_X= \omega^{[r]}_{(X,0)} \ (r\in 2\Z)$ if $X$ has semi-log canonical singularities. 
Among weakly log canonical singularities, semi-log canonical singularities are those which have multiplicity 
$1$ or $2$ in codimension one.

We classify toric varieties $X=\Spec k[\cM]$ which are weakly (semi-) log canonical. The classification is combinatorial, 
expressed in terms of the log structure on the normalization, and certain incidence numbers of the irreducible components 
in their invariant codimension one subvarieties. The irreducible case is much simpler than the 
reducible case. Along the way, we find a criterion for $X$ to satisfy Serre's property $S_2$, which extends Terai's
criterion~\cite{Ter07}.

A key feature of weakly log canonical singularities is the definition of residues onto lc centers of codimension one.
We make this explicit in the toric case. We also construct residues to higher codimension 
lc centers, under the assumption that the irreducible components of the toric variety are normal. In particular, we obtain
higher codimension residues for normal crossings pairs.

We assume the reader is familiar with~\cite[Section 2]{TFR1}, which may be used to construct examples
of weakly normal toric varieties.

We outline the structure of this paper.
In Section 1 we collect known results on log pairs and codimension one residues, and 
exemplify them in the (normal) toric case. In Section 2, we find a criterion (Theorem~\ref{2.10}) 
for $\Spec k[\cM]$ to satisfy Serre's property $S_2$. The irreducible case was known~\cite{BLR06},
and our criterion generalizes that of Terai~\cite{Ter07}.
The weak normality criterion for $\Spec k[\cM]$ was also known (see~\cite{TFR1} for a survey and references).
In Section 3 we define weakly normal log pairs, and the class of weakly log canonical singularities. 
Compared to semi-log canonical pairs, weakly normal log pairs are allowed boundaries with negative coefficients, 
and a certain locus where it is not weakly log canonical. Hopefully, this will be useful in future applications.
In Section 4, we find a criterion for $\Spec k[\cM]$, endowed with a torus invariant boundary $B$, to be 
a weakly normal log pair (Proposition~\ref{IC} for the irreducible case, Proposition~\ref{wlpC} for the reducible case).
We also investigate the $\LCS$-locus, or non-klt locus of a toric weakly normal pair, which is useful for 
inductive arguments. In Section 5 we construct residues of toric weakly log canonical pairs onto lc centers 
of arbitrary codimension, under the assumption that the irreducible components of the toric variety are normal.
We extend these results to weakly log canonical pairs which are locally analytically isomorphic to such toric
models (Theorem~\ref{nwa}). In particular, we obtain higher codimension residues for normal crossings pairs
(Corollary~\ref{nwaNC}).

\begin{ack} I would like to thank Viviana Ene for useful discussions, and the anonymous referee for suggestions 
and corrections.
\end{ack}


\section{Preliminary on log pairs, codimension one residues}



\subsection{Rational pluri-differential forms on normal varieties}


Let $X/k$ be a normal algebraic variety, irreducible, of dimension $d$.
A prime divisor on $X$ is a codimension one subvariety $P$ in $X$.

A non-zero rational function $f\in k(X)^\times$ induces the principal Weil divisor on $X$
$$
(f)=\dv_X(f)=\sum_P v_P(f)\cdot P,
$$ 
where the sum runs after all prime divisors of $X$. Note that $v_P(f)$ is the maximal
$m\in \Z$ such that $t_P^{-m}f$ is regular at $P$, where $t_P$ is a local parameter at $P$.

A non-zero rational differential $d$-form $\omega\in \wedge^d\Omega^1_{k(X)/k}\setminus 0$
induces a Weil divisor on $X$
$$
(\omega)=\sum_P v_P(\omega)\cdot P,
$$
where $v_P(\omega)$ is the maximal $m\in \Z$ such that $t_P^{-m}\omega$ is regular at $P$, where 
$t_P$ is a local parameter at $P$. If $\omega'\in \wedge^d\Omega^1_{k(X)/k}\setminus 0$,
then $\omega'=f\omega$ for some $f\in k(X)^\times$, and $(\omega')=(f)+(\omega)$.
Therefore the linear equivalence class of $(\omega)$ is an invariant of $X$, called the 
{\em canonical divisor} of $X$, denoted $K_X$. Sometimes we also denote by $K_X$ any divisor
in this class, but this may cause confusion.

Let $r\in \Z$. A non-zero rational $r$-pluri-differential form $\omega\in (\wedge^d\Omega^1_{k(X)/k})^{\otimes r}
\setminus 0$ induces a Weil divisor on $X$
$$
(\omega)=\sum_P v_P(\omega)\cdot P,
$$
where if we write $\omega=f\omega_0^r$ with $f\in k(X)^\times$ and $\omega_0\in\Omega^1_{k(X)/k}
\setminus 0$, we define $(\omega)=(f)+r(\omega_0)$.
This is well defined, and $(\omega)\sim rK_X$.

The following properties hold:
$
(f\omega)=(f)+(\omega), \ 
(\omega_1\omega_2)=(\omega_1)+(\omega_2).
$
Note that rational functions identify with rational differential $0$-forms.

Let $P\subset X$ be a prime divisor. A rational differential $p$-form 
$\omega\in \wedge^p \Omega^1_{k(X)/k}$ has {\em at most a
logarithmic pole} at $P$ if both $\omega$ and $d\omega$ have at most a simple
pole at $P$. Equivalently, there exists a decomposition $\omega=(dt/t)\wedge \omega^{p-1}+\omega^p$,
with $t$ a local parameter at $P$, and $\omega^{p-1},\omega^p$ regular at $P$. 
Define the {\em Poincar\'e residue of $\omega$ at $P$} to be the rational differential form
$$
\Res_P \omega=\omega^{p-1}|_P\in \wedge^{p-1} \Omega^1_{k(P)/k}.
$$
The definition is independent of the decomposition. It is additive in $\omega$,
and if $f\in k(X)$ is regular at $P$, then $f|_P\in k(P)$ and $\Res_P(f\omega)=f|_P\cdot \Res_P(\omega)$.

Note that $\omega \in \wedge^d \Omega^1_{k(X)/k}$ automatically satisfies $d\omega=0$. Therefore $\omega$
has at most a logarithmic pole at $P$ if and only if $(\omega)+P\ge 0$ near $P$.


\subsection{Log pairs and varieties}


Let $X/k$ be a normal algebraic variety. Let $B$ be a $\Q$-Weil divisor on $X$:
a formal sum of prime divisors on $X$, with rational coefficients, or equivalently,
the formal closure of a $\Q$-Cartier divisor defined on the smooth locus of $X$.
For $n\in \Z$, define a coherent $\cO_X$-module $\omega^{[n]}_{(X/k,B)}$ by setting for 
each open subset $U\subseteq X$
$$
\Gamma(U,\omega^{[n]}_{(X/k,B)})=\{0\}\cup 
\{\omega\in (\wedge^d\Omega^1_{k(X)/k})^{\otimes n}; (\omega)+nB\ge 0\text{ on } U \}.
$$

On $V=X\setminus (\Sing X\cup \Supp B)$, $\omega^{[n]}_{(X/k,B)}|_V$ coincides with 
the invertible $\cO_V$-module $(\wedge^d\Omega^1_{V/k})^{\otimes n}$.

\begin{lem} Let $U\subseteq X$ be an open subset.
Let $\omega\in (\wedge^d\Omega^1_{k(X)/k})^{\otimes n} \setminus 0$ be a non-zero rational 
pluri-differential form. Then $1\mapsto \omega$ induces an isomorphism 
$\cO_U\isoto \omega^{[n]}_{(X/k,B)}|_U$ if and only if $(\omega)+\lfloor nB\rfloor=0$ on $U$.
\end{lem}

\begin{proof} Indeed, the homomorphism is well defined only if $D=(\omega)+\lfloor nB\rfloor |_U\ge 0$.
The homomorphism is an isomorphism if and only if $\cO_U=\cO_U(D)$, that is $D=0$,
since $U$ is normal.
\end{proof}

The choice of a non-zero rational top differential form on $X$ induces an isomorphism between 
the sheaf of rational pluri-differentials $\omega^{[n]}_{(X/k,B)}$ and the sheaf of rational functions $\cO_X(nK_X+\lfloor nB\rfloor)$.

We have a natural multiplication map
$\omega^{[m]}_{(X/k,B)} \otimes_{\cO_X} \omega^{[n]}_{(X/k,B)} \to \omega^{[m+n]}_{(X/k,B)}$,
which is an isomorphism if $mB$ has integer coefficients and $\omega^{[m]}_{(X/k,B)}$ is invertible. 
In particular, if $rB$ has integer coefficients and $\omega^{[r]}_{(X/k,B)}$ is invertible, then 
$(\omega^{[r]}_{(X/k,B)})^{\otimes n}\isoto \omega^{[rn]}_{(X/k,B)}$ for all $n\in \Z$, and 
the graded $\cO_X$-algebra $\oplus_{n\in \N} \omega^{[n]}_{(X/k,B)}$ is finitely generated.

\begin{defn}
A {\em log pair} $(X/k,B)$ consists of a normal algebraic variety $X/k$ and the (formal) closure $B$ 
of a $\Q$-Weil divisor on the smooth locus of $X/k$, subject to the following property: there exists an 
integer $r\ge 1$ such that $rB$ has integer coefficients and the $\cO_X$-module $\omega^{[r]}_{(X/k,B)}$ 
is locally free (i.e. invertible).

If $B$ is effective, we call $(X/k,B)$ a {\em log variety}.
\end{defn}


\subsection{Log canonical singularities, lc centers}

We assume log resolutions are known to exist (e.g. if $\Char(k)=0$, by Hironaka, or in the
category of toric log pairs). Let $(X/k,B)$ be a log pair. There exists a {\em log resolution} 
$\mu\colon X'\to (X,B_X)$,
that is a desingularization $\mu\colon X'\to X$ such that $\Exc\mu\cup \mu^{-1}(\Supp B)$ is a 
normal crossings divisor. Let $r\ge 1$ such that $rB$ has integer coefficients and $\omega^{[r]}_{(X/k,B)}$
is invertible. If $\omega$ is a local generator, then $\mu^*\omega$ is a local generator of 
$\omega^{[r]}_{(X'/k,B_{X'})}$, where $B_{X'}$ is a $\Q$-divisor on $X'$ such that $rB_{X'}$ has
integer coefficients (locally, $B_{X'}=-\frac{1}{r}(\mu^*\omega)$). 
The $\Q$-divisor $B_{X'}$ may not be effective even if $B$ is effective, and this is the reason
why we consider log pairs, although we are mainly interested with log varieties.

We obtain a log crepant desingularization 
$
\mu\colon (X',B_{X'})\to (X,B),
$
with $X'$ smooth and $\Supp(B_{X'})$ a normal crossings divisor, and an isomorphism
$
\mu^*\omega^{[r]}_{(X/k,B)}\isoto \omega^{[r]}_{(X',B_{X'})}.
$

If the coefficients of $B_{X'}$
are at most $1$, we say that $(X,B)$ has {\em log canonical singularities}. This definition is 
independent of the choice of $\mu$. If $B_{X'}^{>1}$ denotes the part of $B_{X'}$ which has 
coefficients strictly larger than $1$, then $\mu(\Supp (B_{X'}^{>1}))$ is a closed subset of $X$,
called the {\em non-lc locus} of $(X,B)$, denoted $(X,B)_{-\infty}$. It is the complement in $X$
of the largest open subset where $(X,B)$ has log canonical singularities. 
An {\em lc center} of $(X,B)$ is either $X$, or $\mu(E)$ for some prime divisor $E$ on some log resolution 
$X'\to X$, with $\mult_E(B_{X'})=1$ and $\mu(E)\not\subseteq (X,B)_{-\infty}$. If 
$\mu\colon (X',B_{X'})\to (X,B)$ is a log resolution such that $B_{X'}^{=1}$ has simple normal crossings,
the lc centers of $(X,B)$ different from $X$ are exactly the images, not contained in $(X,B)_{-\infty}$, 
of the intersections of the components of $B_{X'}^{=1}$. In particular, $(X,B)$ has only finitely many lc centers.


\subsection{Residues in codimension one lc centers, different}


Let $(X/k,B)$ be a log pair, let $E\subset X$ be a prime divisor with $\mult_E(B)=1$.
Let $\omega\in \Gamma(X,\omega^{[l]}_{(X/k,B)})$. Near the generic point of $E$,  
$\omega^{[1]}_{(X/k,B)}$ is invertible, say with generator $\omega_0$.
We can write $\omega=f\omega_0^{\otimes l}$, with $f\in k(X)^\times$ regular at the generic
point of $E$. Define the {\em residue of $\omega$ at $E$} to be the rational pluri-differential form
$$
\Res^{[l]}_E \omega=f|_E\cdot (\Res_E \omega_0)^{\otimes l}\in (\wedge^{d-1}\Omega^1_{k(E)/k})^{\otimes l}.
$$
The definition is independent of the choice of $f$ and $\omega_0$. 
It is additive in $\omega$, and if $g\in k(X)$ is regular at the generic point of $E$, 
then $\Res^{[l]}_E(g\cdot \omega)=g|_E\cdot \Res^{[l]}_E\omega$.
The residue operation induces a natural map
$$
\Res_E^{[l]}\colon \omega^{[l]}_{(X/k,B)} \to \omega^{[l]}_{k(E)/k},
$$
which is compatible with multiplication of pluri-differential rational forms.

Let $r\ge 1$ such that $rB$ has integer coefficients and $\omega^{[r]}_{(X/k,B)}$ is invertible. 
Let $E^n\to E$ be the normalization and $j\colon E^n \to X$ the induced morphism. 
Choose an open subset $U\subseteq  X$ which intersects $E$, and a nowhere zero section $\omega$ of 
$\omega^{[r]}_{(X/k,B)}|_U$. Then $\Res^{[r]}_E \omega$ is a non-zero rational pluri-differential form on $E^n$. 
The identity 
$$
(\Res^{[r]}_E \omega )|_{j^{-1}(U)}+D|_{j^{-1}(U)}=0
$$ 
defines a Weil divisor $D$ on $j^{-1}(U)$. It does not depend on the choice of $\omega$,
and it glues to a Weil divisor $D$ on $E^n$. The $\Q$-Weil divisor 
$
B_{E^n}=\frac{1}{r}D
$
is called the {\em different of $(X,B)$ on $E^n$}. It follows that $rB_{E^n}$ has integer coefficients,
$\omega^{[r]}_{(E^n/k,B_{E^n})}$ is invertible, and the residue at $E$ induces an isomorphism
$$
\Res^{[r]}_E \colon j^*\omega^{[r]}_{(X/k,B)} \isoto \omega^{[r]}_{(E^n/k,B_{E^n})}.
$$
If $l\ge 1$ is an integer, then $\omega^{[rl]}_{(X/k,B)}$ is again invertible. It defines the same
different, and the isomorphism $\Res^{[rl]}_E$ identifies with $(\Res^{[r]}_E)^{\otimes l}$.
We deduce that the different $B_{E^n}$ is independent of the choice of $r$, and $(E^n/k,B_{E^n})$ is 
again a log pair. The following properties hold:
\begin{itemize}
\item  If $B\ge 0$, then $B_{E^n}\ge 0$.
\item Let $B'$ such that $\mult_E B'=1$ and $B'-B$ is $\Q$-Cartier. Then $B'_{E^n}=B_{E^n}+j^*(B'-B)$.
\end{itemize}


\subsection{Volume forms on the torus}


Let $T/k$ be a torus, of dimension $d$. Then $T=\Spec k[M]$ for some lattice $M$. 
Let $\cB=(m_1,\ldots,m_d)$ be an ordered basis of the lattice $M$. Then 
$$
\omega_\cB=\frac{d\chi^{m_1}}{\chi^{m_1}}\wedge \cdots \wedge \frac{d\chi^{m_d}}{\chi^{m_d}}
$$
is a $T$-invariant global section of $\wedge^d \Omega^1_{T/k}$, which is nowhere zero.
It induces an isomorphism 
$$
\cO_T\isoto \wedge^d \Omega^1_{T/k}.
$$
Let $\cB'=(m'_1,\ldots,m'_d)$ be another ordered basis of $M$. Then 
$\omega_\cB=\epsilon \cdot \omega_{\cB'}$, where the sign $\epsilon= \pm 1$ is computed either
by the identity $\wedge_{i=1}^d m_i=\epsilon \cdot \wedge_{i=1}^d m'_i$ in $\wedge^dM$,
or as the determinant of the matrix $(a_{ij})$ given by $m_i=\sum_j a_{ij}m'_j$. 
Therefore $\omega_\cB$ depends on the choice of the ordered basis of $M$ only up to a
sign. If the sign does not matter, we denote $\omega_\cB$ by $\omega_T$ or $\omega_M$.
For example, if $n$ is an {\em even} integer, we denote $\omega_\cB^{\otimes n}$ by $\omega_T^{\otimes n}$.

The above trivialization of $\wedge^d \Omega^1_{T/k}$ depends on the choice 
of the ordered basis up to a sign. Its invariant form is $\cO_T\otimes_\Z \wedge^dM\isoto \wedge^d\Omega^1_{T/k}$
(induced by $\cO_T\otimes_\Z M \isoto \Omega^1_{T/k}$). The form $\omega_\cB$ depends
in fact only on the basis element $m_1\wedge\cdots \wedge m_d$ of $\wedge^d M\simeq \Z$.
We say that $\omega_\cB$ is the {\em volume form induced by an orientation of $M$}.

Let $M'\subseteq M$ be a sublattice of finite index $e$. It corresponds to a finite surjective
toric morphism $\varphi\colon T=\Spec k[M] \to T'=\Spec k[M']$. If $\cB'$ is an ordered basis of 
$M'$, then $\varphi^*\omega_{\cB'}=(\pm e)\cdot \omega_\cB$.


\subsection{Affine toric log pairs}

Let $T\subseteq X$ be a normal affine equivariant embedding of a torus.
Thus $T=\Spec k[M]$ for some lattice $M$, and $X=\Spec k[M\cap \sigma]$ for
a rationally polyhedral cone $\sigma\subseteq M_\R$ which generates $M_\R$.
The complement $\Sigma_X= X\setminus T$ is called the toric boundary of $X$. We have
$\Sigma_X=\cup_i E_i$, where $E_i$ are the invariant codimension one subvarieties of $X$. 
Each $E_i$ is of the form $\Spec k[M\cap \tau_i]$, where $\tau_i\prec \sigma$ is a 
codimension one face. Let $e_i\in N\cap \sigma^\vee$ be the primitive vector in the dual 
lattice $N$ which cuts out $\tau_i$, that is $\sigma^\vee\cap \tau_i^\perp\cap N=\N e_i$.

The volume form $\omega_\cB$ on $T$, induced by an orientation of $M$,
extends as a rational top differential form on $X$.
Let $E_i$ be an invariant prime divisor on $X$. As a subvariety, $E_i=\Spec k[M\cap \tau_i]$ is again toric and normal. 
Denote by $M_i$ the lattice $M\cap \tau_i-M\cap \tau_i=M\cap (\tau_i-\tau_i)$. Let $\cB_i=(m'_1,\ldots,m'_{d-1})$ be an 
ordered basis of $M_i$. Choose $u\in M$ such that $\langle e_i,u\rangle=1$. Then $\cB_i'=(u,m'_1,\ldots,m'_{d-1})$ 
becomes an ordered basis of $M$, and $\omega_\cB=\epsilon_i\cdot \omega_{\cB'_i}$ for some $\epsilon_i = \pm 1$.
The sign $\epsilon_i$ does not depend on the choice of $u$. Since $\chi^u$ is a local parameter at the generic point of 
$E_i$, and $\omega_{\cB'_i}=\frac{d\chi^u}{\chi^u}\wedge \omega_{\cB_i}$, we obtain
$$
\Res_{E_i} \omega_\cB=\epsilon_i \cdot \omega_{\cB_i}.
$$
Therefore $\omega_\cB$ has exactly logarithmic poles along the invariant prime divisors of $X$, and the induced Weil divisor 
on $X$ is 
$$
(\omega_\cB)=-\Sigma_X.
$$

\begin{lem}
$(X/k,\Sigma_X)$ is a log variety with lc singularities.
\end{lem}

\begin{proof} We have $\omega^{[1]}_{(X/k,\Sigma_X)}=\cO_X\cdot \omega_\cB$, so $\omega^{[1]}_{(X/k,\Sigma_X)}\simeq \cO_X$. 
Let $\mu\colon X'\to X$ be a toric desingularization. Let $\Sigma_{X'}=X'\setminus T$ be the toric boundary
of $X'$, which is the union of its invariant codimension one subvarieties. Then $X'$ is smooth, $\Sigma_{X'}$ is a 
simple normal crossings divisor, and $(\mu^*\omega_\cB)+\Sigma_{X'}=0$. 
Therefore $(X/k,\Sigma)$ has log canonical singularities.
\end{proof}

The different of $(X/k,\Sigma_X)$ on $E_i$ is $\Sigma_{E_i}$, and for every $n\in \Z$ we have residue isomorphisms
$$
\Res^{[n]}_{E_i}\colon \omega^{[n]}_{(X/k,\Sigma_X)}|_{E_i}\isoto \omega^{[n]}_{(E_i/k,\Sigma_{E_i})}.
$$
Choosing bases $\cB,\cB_i$ to trivialize the sheaves, the residue isomorphism becomes
$\epsilon_i^n\cdot (\cO_X|_{E_i}\isoto \cO_{E_i})$.

Let $B$ be a $T$-invariant $\Q$-Weil divisor on $X$. That is $B=\sum_i b_iE_i$ with $b_i\in \Q$. We compute 
$$
\omega^{[n]}_{(X/k,B)}=\cO_X(\lfloor -n\Sigma_X+nB\rfloor)\cdot \omega_\cB ^{\otimes n}.
$$

Recall that $X$ has a unique closed orbit, associated to the smallest face of $\sigma$, which
is $\sigma\cap(-\sigma)$, or equivalently, the largest vector subspace contained in $\sigma$.

\begin{lem}\label{NI} Let $n\in \Z$. The following properties are equivalent:
\begin{itemize}
\item[a)] $\omega^{[n]}_{(X/k,B)}$ is invertible at some point $x$, which belongs to the closed orbit of $X$.
\item[b)] $\omega^{[n]}_{(X/k,B)}\simeq \cO_X$.
\item[c)] There exists $m\in M$ such that $(\chi^m)+\lfloor n(-\Sigma_X+B) \rfloor=0$ on $X$.
\end{itemize}
\end{lem}

\begin{proof} $a) \Longrightarrow c)$ The $T$-equivariant sheaf $\cO_X(\lfloor n(-\Sigma_X+B)\rfloor)$ is invertible near $x$.
Since $x$ belongs to the closed orbit of $X$, the sheaf is trivial, and 
there exists $m\in M$ with $(\chi^m)+\lfloor n(-\Sigma_X+B)\rfloor=0$ on $X$~\cite{KKMS73}. 

$c) \Longrightarrow b)$ $\chi^m\omega_\cB^{\otimes n}$ is a nowhere zero global
section of $\omega^{[n]}_{(X/k,B)}$.

$b) \Longrightarrow a)$ is clear.
\end{proof}

\begin{prop} $(X/k,B)$ is a log pair if and only if and only if there exists 
$\psi\in M_\Q$ such that $\langle e_i,\psi\rangle=1-\mult_{E_i}(B)$ for all $i$.
Moreover, $(X/k,B)$ has lc singularities if and only if the coefficients of $B$ are at most $1$,
if and only if $\psi\in \sigma$.
\end{prop}

\begin{proof} Suppose $(X/k,B)$ is a log pair. There exists $r\ge 1$ such that $rB$ has integer
coefficients and $\omega^{[r]}_{(X/k,B)}$ is invertible. Then there exists $m\in M$ with $(\chi^m)+\lfloor r(-\Sigma_X+B) \rfloor=0$ on $X$.
That is $\langle e_i,m\rangle=r(1-\mult_{E_i}(B))$ for every $i$. Then $\psi=\frac{1}{r}m$ satisfies the desired properties.

Conversely, let $\psi\in M_\Q$ with $\langle e_i,\psi\rangle=1-\mult_{E_i}(B)$ for all $i$. Let
$r\ge 1$ with $r\psi\in M$. In particular, $rB$ has integer coefficients. Since $(\chi^{r\psi})+r(-\Sigma_X+B)=0$,
$\omega^{[r]}_{(X/k,B)}\simeq \cO_X$.

The above proof also shows that $rB$ has integer coefficients and $\omega^{[r]}_{(X/k,B)}$ is invertible
if and only if $r\psi\in M$. 

Suppose $\psi\in \sigma$. Since $\{ \R_+e_i\}_i$ are the extremal rays of $\sigma^\vee$, this is equivalent 
to $\langle e_i,\psi\rangle\ge 0$ for all $i$, which in turn is equivalent to $\mult_{E_i}(B)\le 1$ for all $i$,
that is $B\le \Sigma_X$. Since $(X/k,\Sigma_X)$ has log canonical singularities, so does $(X,B)$.

Conversely, suppose $(X/k,B)$ has log canonical singularities. Then the coefficients of $B$ are at most $1$,
that is $\psi\in \sigma$.
\end{proof}

We call $(X/k,B)$ a {\em toric (normal) log pair}, and $\psi\in M_\Q$ the {\em log discrepancy function} of the 
toric log pair $(X/k,B)$. The log discrepancy function is unique only up to an element of 
$M_\Q\cap \sigma\cap (-\sigma)$. It uniquely determines the boundary, by the formula
$B=\sum_i(1-\langle e_i,\psi\rangle)E_i$. The terminology derives from the following property:

\begin{lem}
Let $(X/k,B)$ be a toric log pair. Each $e\in N^{prim}\cap \sigma^\vee$ induces a toric valuation
$E_e$ over $X$, with log discrepancy $a(E_e;X,B)=\langle e,\psi\rangle$.
\end{lem}

\begin{proof}
Let $\Delta$ be a fan in $N$ which is a subdivision of $\sigma$, and contains $\R_+e$ as a face.
Let $X'=T_N\emb(\Delta)$ be the induced toric variety. Then $\mu\colon X'\to X$ is a toric proper
birational morphism, and $e$ defines an invariant prime $E_e$ on $X'$.
Let $r\psi\in M$. Then $\chi^{r\psi}\omega_\cB^{\otimes r}$ trivializes $\omega^{[r]}_{(X/k,B)}$,
hence $\mu^*\chi^{r\psi}\omega_\cB^{\otimes r}$ trivializes $\omega^{[r]}_{(X'/k,B_{X'})}$.
Therefore $1-\mult_{E_e}(B_{X'})=\langle e,\psi \rangle$.
\end{proof}

We have $(X/k,B)_{-\infty}=\cup_{b_i>1}E_i$. If non-empty, the non-lc locus has pure codimension one in $X$.
If $B$ is effective, the non-lc locus is the support of a natural subscheme structure~\cite{Amb14},
with ideal sheaf $I_{-\infty}=\oplus_m k\cdot \chi^m$, where the sum runs after all $m\in M\cap \sigma$
such that $\langle m,e\rangle\ge \max(0,-\langle \psi,e\rangle)$ for all $e\in N\cap \sigma^\vee$.
From the existence of toric log resolutions, it follows that the lc centers of $(X/k,B)$ are the invariant 
subvarieties $X_\sigma$, where $\psi\in \sigma\prec \sigma_S$ and $\sigma\not\subset \tau_i$ if $b_i>1$.

Let $(X/k,B)$ be a toric log pair, with log canonical singularities. That is $\psi\in M_\Q\cap \sigma$.
The lc centers of $(X/k,B)$ are the invariant closed irreducible subvarieties
$X_\tau=\Spec k[M\cap \tau]$, where $\tau$ is a face of $\sigma$ which contains $\psi$. For $\tau=\sigma$,
we obtain $X$ as an lc center. For $\tau\ne \sigma$, we obtain lc centers defined by toric valuations of $X$.
Each lc center is normal. Any union of lc centers is weakly normal. The intersection of two lc centers is again
an lc center. With respect to inclusion, there exists a unique minimal lc center, namely $X_{\tau(\psi)}$
for $\tau(\psi)=\cap_{\psi\in \tau\prec \sigma}\tau$ (the unique face of $\sigma$ which contains $\psi$ in its relative
interior). Note that $X$ is the unique lc center of $(X/k,B)$ if and only
if $(X/k,B)$ has klt singularities, if and only if the coefficients of $B$ are strictly less than $1$, if and only if 
$\psi$ belongs to the relative interior of $\sigma$.
Define the {\em LCS locus}, or {\em non-klt locus} of $(X/k,B)$, to be the union of all lc centers of positive codimension in $X$.
We have $\max_i b_i\le 1$ and 
$
\LCS(X/k,B)=\cup_{b_i=1}E_i.
$
If non-empty, the $\LCS$ locus has pure codimension one in $X$.

Let $(X/k,B)$ be a toric log pair, let $E_i$ be an invariant prime divisor with $\mult_{E_i}(B)=1$.
Let $\psi$ be the log discrepancy function, let $r\psi\in M$. We have $E_i=\Spec k[M\cap \tau_i]$ 
for a codimension one face $\tau_i\prec \sigma$.
The condition $\mult_{E_i}(B)=1$ is equivalent to $r\psi\in M_i=M\cap \tau_i-M\cap \tau_i$.
Then $\chi^{r\psi}\omega_\cB^{\otimes r}$ trivializes $\omega^{[r]}_{(X/k,B)}$, and
$$
\Res_{E_i}(\chi^{r\psi}\omega_\cB^{\otimes r})=\chi^{r\psi}(\Res_{E_i} \omega_\cB )^{\otimes r}=
\epsilon_i^r \chi^{r\psi} \omega_{\cB_i}^{\otimes r}
$$
trivializes $\omega^{[r]}_{(E_i/k,B_{E_i})}$, where $B_{E_i}$ is the different of $(X/k,B)$ on $E_i$,
computed by the formula 
$$
(\chi^{r\psi})=r(\Sigma_{E_i}-B_{E_i})  \text{ on } E_i.
$$
Let $n\in \Z$. Then $\Res^{[n]}_{E_i}$ sends $\omega^{[n]}_{(X/k,B)}$ into $\omega^{[n]}_{(E_i/k,B_{E_i})}$.
If $\omega^{[n]}_{(X/k,B)}$ is invertible (even if $nB$ does not have integer coefficients), we obtain an isomorphism 
$$
\Res^{[n]}_{E_i} \colon \omega^{[n]}_{(X/k,B)}|_{E_i}\isoto \omega^{[n]}_{(E_i/k,B_{E_i})}.
$$
The coefficients of the different $B_{E_i}$ are controlled by those of $B$. Indeed, let $Q\subset E_i$ be 
an invariant prime divisor. The lattice dual to $M_i$ is a 
quotient lattice $N_i$ of $N$, and the cone in $N_i$ dual to $\tau_i\subset (M_i)_\R$ is the image of $\sigma^\vee\subset N_\R$ 
under the quotient $\pi\colon N \to N_i$. Let $e_Q\in N_i$ be the primitive vector on the extremal ray of the cone dual to $\tau_i$,
which determines $Q\subset E_i$. There exists an extremal ray of $\sigma^\vee$ which maps onto $\R_+e_Q$, and denote
by $e_j$ its primitive vector. Then $\pi(e_j)=q e_Q$ for some positive integer $q$. 
Since $\langle e_j,\psi\rangle= q \langle e_Q,\psi\rangle$, we obtain 
$$
\mult_Q(B_{E_i})=1-\frac{1-\mult_{E_j}(B)}{q}.
$$


\section{Serre's property $S_2$ for affine toric varieties}


Let $X=\Spec k[\cM]$ be the affine toric variety associated to a monoidal complex
$\cM=(M,\Delta,(S_\sigma)_{\sigma\in \Delta})$. The reader may consult~\cite[Section 2]{TFR1}
for basic definitions, notations and properties of such (possibly not normal or irreducible) toric varieties
$X$, including the criteria for $X$ to be seminormal or weakly normal.

The torus $T=\Spec k[M]$ acts naturally on $X$.
We give a combinatorial criterion for $X$ to satisfy Serre's property $S_2$.
Note that $X$ is irreducible if and only if $\Delta$ has a unique maximal element,
if and only if $X=\Spec k[S]$, where $S\subseteq M$ is a finitely generated semigroup such that 
$S-S=M$.


\subsection{Irreducible case}

Let $S\subseteq M$ be a finitely generated semigroup such that $S-S=M$. 
Let $k[S]$ be the induced semigroup ring, set $X=\Spec k[S]$. It is an equivariant embedding
of $T$. Let $\sigma_S\subseteq M_\R$ be the cone generated by $S$. For a face $\sigma\prec \sigma_S$,
denote $S_\sigma=S\cap \sigma$ and $X_\sigma=\Spec k[S_\sigma]$. Then $X$ is the toric variety
associated to the monoidal complex determined by $M$, the fan $\Delta$ consisting of faces of $\sigma_S$,
and the collection of semigroups $S_\sigma$. The invariant closed subvarieties of $X$ are $X_\sigma$.

If $S=M$, then $T=X$ is smooth, hence $S_2$. Else, $X\setminus T=\sum_i E_i$
is the sum of $T$-invariant codimension one subvarieties. We have $E_i=X_{\tau_i}$, where 
$(\tau_i )_i$ are the codimension one faces of $\sigma_S$. Set 
$$
S'=\cap_i (S-S\cap \tau_i).
$$

\begin{lem}
$S\subseteq S'\subseteq \bar{S}=M\cap \sigma_S$.
\end{lem}

\begin{proof} We only have to prove the inclusion $S'\subseteq M\cap\sigma_S$. 
Suppose by contradiction that
$m\in S'\setminus \sigma_S$. Then there exists $\varphi\in \sigma_S^\vee$ such
that $\sigma_S\cap \varphi^\perp$ is a codimension one face $\tau_i$ of $\sigma_S$, and 
$\langle \varphi,m\rangle<0$. But $m+s_i\in S$ for some $s_i\in S\cap {\tau_i}$.
Therefore $\langle \varphi,m\rangle=\langle \varphi,m+s_i\rangle\ge 0$,
a contradiction.
\end{proof}

\begin{thm}\cite{GSW76}
The $S_2$-closure of $X$ is $\Spec k[S']\to \Spec k[S]$.
\end{thm}

\begin{proof} Denote $R=\{f\in k(X); \text{regular in codimension one on }X\}$.
If $f\in R$, then $f|_T$ is regular in codimension one on $T$. Since $T$ is normal, 
$f$ is regular on $T$. Therefore $R\subseteq \cO(T)=\oplus_{m\in M} k\cdot \chi^m$.
Now $R$ is $T$-invariant. Therefore $R=\oplus_{m\in S_1} k\cdot \chi^m$, for a 
certain semigroup $S_1\subseteq M$ which remains to be identified.

Let $m\in S_1$. Let $\tau_i\prec \sigma_S$ be a face of codimension one. Then 
$\chi^m$ is regular at the generic point of $E_i$. That is $m=s-s'$,
for some $s\in S$ and $s'\in S\cap \tau_i$. We deduce that 
$S_1\subseteq \cap_i (S-S\cap \tau_i)=S'$. For the converse, let $m\in S'$.
Since $\chi^m$ is regular on $T$ and at the generic points of $X\setminus T=\sum_i E_i$, it is regular 
in codimension one on $X$. Therefore $m\in S_1$.
\end{proof}

In particular, $\Spec k[S]$ is $S_2$ if and only if $S=\cap_i (S-S\cap \tau_i)$.

Recall~\cite[Proposition 2.10]{TFR1} that $X=\Spec k[S]$ is seminormal if 
and only if $S=\sqcup_{\sigma \prec \sigma_S} \Lambda_\sigma\cap \relint \sigma$,
where $(\Lambda_\sigma)_{\sigma\prec \sigma_S}$ is a family of sublattices of
finite index $\Lambda_\sigma\subseteq M\cap \sigma-M\cap \sigma$, such that
$\Lambda_{\sigma_S}=M$ and $\Lambda_{\sigma'}\subseteq \Lambda_{\sigma}\cap \sigma'$
if $\sigma'\prec \sigma$. A family of sublattices defines $S$ by the above
formula, and $S$ determines the family of sublattices $\Lambda_\sigma=S\cap \sigma-S\cap \sigma$.

\begin{thm} \cite{BLR06}
$\Spec k[S]$ is seminormal and $S_2$ if and only if 
$\Lambda_\sigma=\cap_{\tau_i\supset\sigma}\Lambda_{\tau_i}$,
for every proper face $\sigma \precnsim \sigma_S$.
\end{thm}

\begin{proof} Recall that $(\tau_i)_i$ are the codimension one faces of $\sigma_S$.
For the proof, we may suppose $\Spec k[S]$ is seminormal.

Suppose $\Spec k[S]$ is not $S_2$. There exists $m\notin S$ and 
$s_i\in S_{\tau_i}\ (1\le i\le q)$ such that $m+s_i\in S$ for all $i$. It follows that $m\in \sigma_S$.
Let $\sigma \prec \sigma_S$ be the unique face which contains $m$ in its relative interior.
Let $\tau_i$ be a codimension one face which contains $\sigma$. Then $m+s_i\in S_{\tau_i}$.
Therefore $m\in S_{\tau_i}-S_{\tau_i}=\Lambda_{\tau_i}$. We obtain 
$m\in \cap_{\tau_i \supset \sigma}\Lambda_{\tau_i}\setminus \Lambda_\sigma$.
Therefore $\Lambda_\sigma$ is strictly contained in $\cap_{\tau_i\supset \sigma}\Lambda_{\tau_i}$.

Conversely, suppose $\Spec k[S]$ is $S_2$. Let $\sigma \precnsim \sigma_S$ be a proper face.
We have an inclusion of lattices $\Lambda_\sigma\subseteq \cap_{\tau_i\supset \sigma}\Lambda_{\tau_i}$,
both generating $\sigma-\sigma$. The inclusion of lattices is an equality, if it is so after restriction to 
$\relint \sigma$, by~\cite[Lemma 2.9]{TFR1}.

Let $m\in \cap_{\tau_i\supset \sigma}\Lambda_{\tau_i}\cap \relint \sigma$.  
If $\tau_i \supseteq \sigma$, then $m\in \Lambda_{\tau_i}\subset S-S_{\tau_i}$. 
If $\tau_i \not\supseteq \sigma$, there exists $s_i \in S\cap \tau_i$ such that $m+s_i \in \Int\sigma_S$. 
Therefore $m+s_i \in M\cap \Int \sigma_S$, which is contained in $S$ by seminormality. 
We obtain $m\in S'$. The $S_2$ property implies that $m\in S$. 
Therefore $m\in \Lambda_\sigma$. We obtain $\Lambda_\sigma=\cap_{\tau_i\supset \sigma}\Lambda_{\tau_i}$.
\end{proof}

So to give $X=\Spec k[S]$ which is seminormal and $S_2$, is equivalent to give $(M,\sigma_S)$
(i.e. the normalization), the codimension one faces $(\tau_i)_i$ of $\sigma_S$, and finite
index sublattices $\Lambda_i\subseteq M\cap \tau_i-M\cap \tau_i$, for each $i$. 
Moreover, $X$ is weakly normal if and only if $\Char(k)$ does not divide the index of 
the sublattice $\Lambda_i \subseteq M\cap \tau_i-M\cap \tau_i$ for all $i$, if and only if 
$\Char(k)$ does not divide the incidence numbers $d_{Y\subset X}$ for every invariant
subvariety $Y$ of $X$ (with the terminology of Definition~\ref{inde}).

The normalization of $X$ is $\bar{X}=\Spec k[\bar{S}]\to X=\Spec k[S]$.
If $X$ is seminormal, the conductor subschemes $C\subset X$ and $\bar{C}\subset \bar{X}$ 
are reduced, described as follows.

\begin{lem} Suppose $X=\Spec k[S]$ is seminormal. Let $\Delta$ be the fan consisting 
of the cones $\sigma\prec \sigma_S$ such that $S_\sigma-S_\sigma\subsetneq 
M\cap \sigma-M\cap \sigma$.
Then $C=\cup_{\sigma\in \Delta}X_\sigma$ and $\bar{C}=\cup_{\sigma\in \Delta}\bar{X}_\sigma$.
\end{lem}

\begin{proof} Note that $\bar{S}_\sigma=M\cap \sigma$ for $\sigma\prec \sigma_S$.

If $S_\sigma-S_\sigma\subsetneq M\cap \sigma-M\cap \sigma$, 
the same property holds for all faces $\tau\prec \sigma$. Therefore $\Delta$ is a fan. 
The conductor ideal is $I=\oplus_{m+\bar{S}\subseteq S} k\cdot \chi^m$. We claim
$$
\{m\in S;  m+\bar{S}\subseteq S\}= S\setminus \cup_{\sigma\in \Delta} \sigma.
$$

For the inclusion $\subseteq$, let $m\in \sigma\prec \sigma_S$
with $m+\bar{S}\subset S$. Then $m+\bar{S}_\sigma \subseteq S_\sigma$. Since
$m\in S_\sigma$, we obtain $S_\sigma-S_\sigma=\bar{S}_\sigma-\bar{S}_\sigma$. 
Therefore $\sigma\notin \Delta$. 

For the inclusion $\supseteq$, let $m\in S$ with $m+\bar{S}\nsubseteq S$.
There exists $\bar{s}\in \bar{S}$ such that $m+\bar{s}\notin S$. Let $\sigma\prec \sigma_S$
be the unique face with $m+\bar{s}\in \relint \sigma$. Then $m,\bar{s}\in \sigma$.
Suppose by contradiction that $\sigma\notin \Delta$. Then $S_\sigma-S_\sigma=
\bar{S}_\sigma-\bar{S}_\sigma$, and 
$$
m+\bar{s}\in \bar{S}_\sigma\cap \relint \sigma=(\bar{S}_\sigma-\bar{S}_\sigma)\cap \relint \sigma=
(S_\sigma-S_\sigma)\cap \relint \sigma=S_\sigma\cap \relint \sigma,
$$
where we have used that $\bar{X}$ and $X$ are seminormal. Then $m+\bar{s}\in S$, a contradiction.
Therefore $\sigma\in \Delta$.
\end{proof}

\begin{defn}\label{inde} Let $X=\Spec k[S]$ and $Y\subset X$ an invariant closed irreducible subvariety.
That is $Y=X_\tau$ for some face $\tau\prec \sigma_S$. Let $\pi\colon \bar{X}\to X$
be the normalization, let $\bar{Y}=\pi^{-1}(Y)$. Then 
$\bar{X}=\Spec k[M\cap \sigma_S]$, $\bar{Y}=(\bar{X})_\tau=\Spec k[M\cap \tau]$ and we obtain
a cartesian diagram
\[ 
\xymatrix{
  \bar{X} \ar[d]_\pi & \bar{Y} \ar@{_{(}->}[l]  \ar[d]^{\pi'}    \\
    X                        &  Y  \ar@{_{(}->}[l]  
} \]
The induced morphism $\pi'\colon \bar{Y}\to Y$ is finite, of degree $d_{Y\subset X}$,
equal to the index of the sublattice $S_\tau-S_\tau\subseteq M\cap \tau-M\cap \tau$.
We call $d_{Y\subset X}$ the {\em incidence number of $Y\subset X$}, sometimes
denoted $d_{\tau\prec \sigma_S}$. Note that $d_{Y\subset X}>1$ if  and only if $X$ is not
normal at the generic point of $Y$.
\end{defn}


\subsection{Reducible case}

Consider now the general case of an affine toric variety $X=\Spec k[\cM]$.
For $\sigma\in \Delta$, denote by $X_\sigma$ the $T$-invariant 
closed irreducible subvariety of $X$ coresponding to $\sigma$. The decomposition of $X$
into irreducible components is $X=\cup_F X_F$, where $\{F\}$ are the facets (maximal 
cones) of $\Delta$. 

\begin{lem}\label{rpr}
The sequence 
$
0\to \cO_X\to \oplus_F \cO_{X_F} \to \oplus_{F\ne F'}\cO_{X_{F\cap F'}}
$
is exact.
\end{lem}

\begin{proof} Let $f_F\in \cO(X_F)$ such that for every $F\ne F'$, $f_F$ and $f_{F'}$
coincide on $X_{F\cap F'}$. We can write $f_F=\sum_m c_m^F\chi^m$.
Let $m\in |\cM |$. The map $F\ni m\mapsto c_m^F$ is constant. Denote by $c_m$
this common value. Then $f=\sum_m c_m\chi^m\in \cO(X)$ and $f|_{X_F}=f_F$ for 
every facet $F$. This shows that the sequence is exact in the middle. The map 
$\cO(X)\to \oplus_F\cO(X_F)$ is clearly injective.
\end{proof}

The $S_2$-closure of $X$ is $\Spec R\to X$, where $R=\varinjlim_{\codim(Z\subset X)\ge 2}\cO_X(X\setminus Z)$
is the ring of functions which are regular in codimension one points of $X$. We describe $R$
explicitly. For $\sigma\in \Delta$, recall that $O_\sigma\subset X_\sigma$ is the open dense orbit.
We have $\sqcup_F O_F\subset X$, with complement $\Sigma=\cup_{\codim(\sigma\in \Delta)>0}X_\sigma$,
the toric boundary of $X$.

Let $f\in R$. Then $f_F:=f|_{O_F}$ is regular in codimension one. Since $T_F$ is normal, hence $S_2$,
$f_F\in \cO(O_F)$. We can uniquely write $f_F=\sum_{m\in S_F-S_F} c_m^F\chi^m$, where the sum has finite
support. Denote $\Supp(f_F)=\{m\in S_F-S_F; c_m^F\ne 0 \}$.

Let $\sigma\in \Delta$ be a cone of codimension one. Equivalently, $\sigma$ has codimension 
one in every facet containing it. Since $f$ is regular at the generic point of $X_\sigma$, we obtain:
\begin{itemize}
\item[1)] $f_F$ is regular at the generic point of $X_\sigma\hookrightarrow X_F$. That is
$\Supp f_F\subset S_F-S_\sigma$.
\item[2)] If $F$ and $F'$ are two facets containing $\sigma$, the restriction of $f_F$ to $X_\sigma\hookrightarrow X_F$
coincides with the restriction of $f_{F'}$ to $X_\sigma\hookrightarrow X_{F'}$.
\end{itemize}
So $f\in R$ induces a family $(f_F)_F\in \prod_F \cO(O_F)$ satisfying properties 1) and 2). This correspondence is bijective,
by Lemma~\ref{rpr}. Thus we may identify $R$ with the collections $(f_F)_F\in \prod_F \cO(O_F)$ satisfying properties 
1) and 2), for every cone $\sigma\in \Delta$ of codimension one. 

\begin{defn}
The fan $\Delta$ is called {\em $1$-connected} if for every two facets $F\ne F'$ of $\Delta$, 
there exists a sequence of facets $F_0=F,F_1,\ldots,F_n=F'$ of $\Delta$, which contain $F\cap F'$, 
and such that $F_i\cap F_{i+1}$ is a face of codimension one in both $F_i$ and $F_{i+1}$, for all $0\le i<n$.
\end{defn}

It is clear that for a $1$-connected fan, every facet has the same dimension.

\begin{lem}\label{1i}
If $X$ is $S_2$, then $\Delta$ is $1$-connected.
\end{lem}

\begin{proof} Let $F\ne F'$ be two facets of $\Delta$. Define a graph $\Gamma$ as follows:
the vertices are the facets of $\Delta$ which contain $F\cap F'$, and two vertices are joined
by an edge if their intersection has codimension one in both of them. Let $\{c\}$ be the 
connected components of $\Gamma$. Denote by $X^{c}$ the union of the irreducible components 
of $X$ which belong to $c$, and $Z=\cup_{c\ne c'} X^c\cap X^{c'}$. By construction,
$\codim(Z\subset X)\ge 2$. Let $Y$ be the union of the irreducible components of $X$ which do not
contain $X_{F\cap F'}$, set $U=X\setminus Y$. 

If $X$ is $S_2$, then $\cO_X(U)\to \cO_X(U\setminus Z)$ 
is an isomorphism. Since $U$ is connected, it follows that $U\setminus Z=\sqcup_c (X^c\setminus Y)$ 
is connected, that is $\Gamma$ is connected. Therefore $F$ and $F'$ can be joined by a chain with the desired properties.
\end{proof}

\begin{lem}\label{2i} 
Suppose $\Delta$ is $1$-connected. Denote by $S'_F$ the $S_2$-closure of $S_F$. 
For $\sigma\in \Delta$, define $\tilde{S}_\sigma=\{m\in \sigma;m\in S'_F\ \forall F\ni m\}$.
Then $\tilde{\cM}=(M,\Delta,(\tilde{S}_\sigma)_{\sigma\in \Delta})$ is a monoidal complex,
and $\Spec k[\tilde{\cM}] \to \Spec k[\cM]$ is the $S_2$-closure of $X$.
\end{lem}

\begin{proof} Since $\Delta$ is $1$-connected, the irreducible components of $X$ have the same dimension.
Therefore $R$ is the ring of collections $(f_F)\in \prod_F \cO(O_F)$ satisfying the following properties:
\begin{itemize}
\item[1')] $f_F$ is regular in codimension one on $X_F$. Since $X_F=\Spec k[S_F]$, this means that 
$\Supp f_F\subset S_F'$.
\item[2')] If $F$ and $F'$ are two facets intersecting in a codimension one face, the restriction of 
$f_F$ to $X_{F\cap F'}\hookrightarrow X_F$ coincides with the restriction of $f_{F'}$ to $X_{F\cap F'} \hookrightarrow X_{F'}$.
\end{itemize} 
Since $\Delta$ is $1$-connected, 2') is equivalent to 
\begin{itemize}
\item[2'')] If $F\ne F'$ are two facets, the restriction of 
$f_F$ to $X_{F\cap F'}\hookrightarrow X_F$ coincides with the restriction of $f_{F'}$ to $X_{F\cap F'} \hookrightarrow X_{F'}$.
\end{itemize} 
Let $m\in \cup_F \Supp f_F$. The map $F\ni m\mapsto c^F_m$ is constant, by 2''). And if the constant $c_m$ is non-zero,
then $m$ belongs to $\cap_{F\ni m}S'_F$, by 1). If we set $f=\sum_m c_m\chi^m$, we have $f|_{X_F}=f_F$ for all $m$.
We conclude that $R$ identifies with the ring of finite sums $\sum_m c_m\chi^m$, such that $c_m\ne 0$ implies 
$m\in \cap_{F\ni m}S'_F$.
 
Denote 
$
\cS=\{ \cap_{i=1}^n F_i;n\ge 1,F_i\in \Delta \text{ facets} \}.
$
The facets of $\Delta$ belong to $\cS$, and if $\sigma,\tau\in \cS$, then $\sigma\cap \tau\in \cS$.
Note that $\cS$ may not contain faces of its cones.
For $\sigma\in \cS$, denote $B\sigma=\cup_{\sigma\supsetneq \tau \in \cS}\tau$. We have 
$$
\cup_{\sigma\in \Delta}\sigma=\cup_F F=\sqcup_{\tau \in \cS}\tau \setminus B\tau. 
$$
If $m\in \cup_F F$, then $\cap_{F\ni m}F$ is the unique element $\tau\in \cS$ such that
$m\in \tau \setminus B\tau$. If $\tau\in \cS$ and $m\in \tau \setminus B\tau$, then 
$
\{F;F\ni m \} = \{F; F\supseteq \tau \}.
$
Therefore $R$ is the toric face ring of the monoidal complex $\tilde{\cM}=(M,\Delta,(\tilde{S}_\sigma)_{\sigma\in \Delta})$,
where 
$$
\tilde{S}_\sigma = \bigsqcup_{\tau \in \cS} \sigma \cap (\tau\setminus B\tau)\cap \bigcap_{F\supseteq \tau}S'_F 
                          = \{m\in \sigma;m\in S'_F\ \forall F\ni m\}.
$$
\end{proof}

Putting Lemmas~\ref{1i} and~\ref{2i} together, we obtain the $S_2$-criterion for $X=\Spec k[\cM]$,
which generalizes Terai's $S_2$-criterion for Stanley-Reisner rings associated to simplicial 
complexes~\cite{Ter07}.

\begin{thm}\label{2.10} $X$ is $S_2$ if and only if the following properties hold:
\begin{itemize}
\item[1)] $\Delta$ is $1$-connected, and 
\item[2)] $S_\sigma=\{m\in \sigma;m\in S'_F\ \forall F\ni m\}$ for every $\sigma\in \Delta$,
where $S'_F$ is the $S_2$-closure of the semigroup $S_F$.
\end{itemize}
\end{thm}

\begin{cor} Suppose each irreducible component of $X$ is $S_2$.
Then $X$ is $S_2$ if and only if $\Delta$ is $1$-connected.
\end{cor}

\begin{cor}\label{4.11}
Suppose $X$ is seminormal, with lattice collection $\Lambda_\sigma=S_\sigma-S_\sigma$. 
Then $X$ is $S_2$ if and only if the following properties hold:
\begin{itemize}
\item[1)] $\Delta$ is $1$-connected.
\item[2)] $\Lambda_\sigma=\cap_{\sigma\subset \tau, \codim(\tau\in \Delta)=1}\Lambda_\tau$
for every $\sigma\in \Delta$ of positive codimension.
\end{itemize}
\end{cor}

So to give $X=\Spec k[\cM]$ which is seminormal and $S_2$, it is equivalent to give
the lattice $M$, a $1$-connected fan $\Delta$ in $M$, sublattices of finite index
$\Lambda_F\subseteq M\cap F-M\cap F$ for each facet $F$ of $\Delta$, and 
sublattices of finite index
$\Lambda_\tau \subseteq M\cap \tau -M\cap \tau$ for each cone $\tau$ of $\Delta$
of codimension one (subject to the condition $\Lambda_\tau\subseteq  \Lambda_F\cap \tau-\Lambda_F\cap \tau$
if $\tau\prec F$). Moreover, $X/k$ is weakly normal if and only if $\Char(k)$ does not divide
the incidence numbers $d_{X_\tau\subset X_F} \ (\tau\prec F)$.

Let $\pi\colon \bar{X}\to X$ be the normalization. Then $\bar{X}=\sqcup_F \bar{X}_F$,
where the direct sum is over all facets of $\Delta$, $\bar{X}_F=\Spec k[\bar{S_F}]$ and
$\bar{S_F}=(S_F-S_F)\cap F$. We compute the conductor ideal $I$ of $\pi$. 
The normalization induces an inclusion of $k$-algebras 
$$
k[\cup_{\sigma\in \Delta}S_\sigma]\to \prod_F k[\bar{S_F}], f\mapsto (f|_{\bar{X_F}})_F.
$$ 
The ideal $I$ consists of $f\in \cO(X)$ such that 
$f\cdot \cO(\bar{X})\subseteq \cO(X)$. It is $T$-invariant, hence of the form 
$$
I=\oplus_{m\in \cup_\sigma S_\sigma\setminus A} k\cdot \chi^m
$$
for a certain set $A$ which it remains to identify. Now $\chi^m\in I$ if and only if 
$\chi^m\cdot (f_F,0,\ldots,0)\in \cO(X)$ for every facet $F$ and every $f_F\in \cO(\overline{X_F})$;
if and only if, for every facet $F$, 
$$
\chi^m\cdot (\chi^a,0,\ldots,0)=(\chi^m|_F\cdot \chi^a,0,\ldots,0)\in \cO(X)
$$ 
for every $a\in \bar{S_F}$; if and only if $m+a\in S_F\setminus \cup_{F'\ne F}F\cap F'$, for all $F\ni m$ 
and $a\in \bar{S_F}$. Therefore 
$$
 \cup_\sigma S_\sigma\setminus A= \{m; m+\bar{S_F}\subset S_F\setminus BF\ \forall F\ni m \}.
$$

If $X$ is seminormal, the conductor subschemes $C\subset X$ and $\bar{C}\subset \bar{X}$ 
are reduced, described as follows.

\begin{lem}\label{4.12}
Suppose $X$ is seminormal. Let $\Delta'$ be the subfan of cones $\sigma\in \Delta$ which 
either are contained in at least two facets of $\Delta$, or are contained in a unique facet $F$ of $\Delta$
and $S_\sigma-S_\sigma\subsetneq (\bar{S_F})_\sigma-(\bar{S_F})_\sigma$.
Then $C=\cup_{\sigma\in\Delta'} X_\sigma$ and 
$\bar{C}=\sqcup_F \cup_{\sigma\in \Delta',\sigma\prec F} (\bar{X}_F)_\sigma$.
\end{lem}

\begin{proof} It suffices to show that the ideal $I$ is radical, hence equal to the ideal of 
union $\cup_{\sigma\in \Delta'}X_\sigma\subset X$.
Indeed, let $m+\bar{S_F}\subset S_F\setminus BF$ for all $F\ni m$.
Assuming $m\in S_\sigma$ for some $\sigma\in \Delta'$, we derive a contradiction.
We have two cases: suppose $\sigma$ is contained in at least two facets $F\ne F'$.
Then $m\in BF$, a contradiction.
Suppose $\sigma$ is contained in a unique facet $F$. Then $m+(\bar{S_F})_\sigma 
\subset (S_F)_\sigma=S_\sigma$. Then $S_\sigma-S_\sigma= 
(\bar{S_F})_\sigma-(\bar{S_F})_\sigma$, that is $\sigma\notin\Delta'$. Contradiction!

Conversely, let $m\in \cup_{\sigma\in \Delta}S_\sigma\setminus \cup_{\sigma\in \Delta'}\sigma$.
Let $m\in F$ be a facet. We must show $m+\bar{S_F}\subset S_F\setminus BF$.
Indeed, let $\bar{s}\in \bar{S_F}$. Then $m+\overline{s}\in \relint \sigma$ for a unique
face $\sigma\prec F$. It follows that $m,\bar{s}\in \sigma$.
If $m+\bar{s}\in BF$, then $m+\bar{s}\in F\cap F'$ for some $F'\ne F$. Then $m,\bar{s}\in F\cap F'$.
Then $m\in F\cap F'\in \Delta'$, a contradiction. Therefore $m+\bar{s}\notin BF$.
On the other hand, $\sigma\notin \Delta'$, that is $S_\sigma-S_\sigma= 
(\bar{S_F})_\sigma-(\bar{S_F})_\sigma$. As in the irreducible case, the seminormality of $X_F$ and its
normalization implies $m+\bar{s} \in S_\sigma\cap \relint \sigma$. Therefore $m+\bar{s} \in S_F$.
\end{proof}


\subsection{The core}

Let $X=\Spec k[\cM]$ be seminormal and $S_2$. Define the {\em core of $X$} to be $X$ if $X$ is normal, 
and the intersection of the irreducible components of the non-normal locus $C$, if $X$ is not normal.

\begin{prop}\label{coredef}
The core of $X$ is normal.
\end{prop}

\begin{proof} Let $\{F\}$ and $\{\tau_i\}$ be the facets and codimension one faces of $\Delta$,
respectively. The core of $X$ is the invariant closed subvariety $X_{\sigma(\Delta)}$, where 
$$
\sigma(\Delta)=\bigcap_F F\cap \bigcap_{X_{\tau_i}\subset C}\tau_i.
$$
Indeed, if $X$ is normal, $\Delta$ has a unique facet $F$ and $C=\emptyset$, hence $\sigma(\Delta)=F$.
If $X$ is not normal, each facet contains some irreducible component of $C$, hence 
$
\sigma(\Delta)=\cap_{X_{\tau_i}\subset C}\tau_i.
$

We claim that $S_{\sigma(\Delta)}=\cap_F(S_F-S_F)\cap \bigcap_{X_{\tau_i}\subset C}(S_{\tau_i}-S_{\tau_i})
\cap \sigma(\Delta)$.
Indeed, the inclusion $\subseteq$ is clear. For the converse, let $m$ be an element on the right hand side.
Then $m\in \relint \tau$ for same $\tau\prec \sigma(\Delta)$.
Let $\tau_i$ be a codimension one face which contains $\tau$. If $X_{\tau_i}\subset C$, then $m\in S_{\tau_i}-S_{\tau_i}$
by assumption. If $X_{\tau_i}\not\subset C$, there exists a unique facet $F$ which contains $\tau_i$, and 
$S_{\tau_i}-S_{\tau_i}=(S_F-S_F)\cap \tau_i-(S_F-S_F)\cap \tau_i$. By assumption, $m$ belongs to the right
hand side. We conclude that $m\in \cap_{\tau_i\succ \tau}S_{\tau_i}-S_{\tau_i}$. By Corollary~\ref{4.11},
this means $m\in S_\tau-S_\tau$. Then $m$ belongs to $(S_\tau-S_\tau)\cap \relint \tau$, which is contained in
$S_\tau$ since $X$ is seminormal. Therefore $m\in S_\tau$, hence $m\in S_{\sigma(\Delta)}$.

From the claim, $S_{\sigma(\Delta)}$ is the trace on $\sigma(\Delta)$ of some lattice. Therefore
$$
S_{\sigma(\Delta)}-S_{\sigma(\Delta)}=\cap_F(S_F-S_F)\cap \bigcap_{X_{\tau_i}\subset C}(S_{\tau_i}-S_{\tau_i}).
$$
and $S_{\sigma(\Delta)}=(S_{\sigma(\Delta)}-S_{\sigma(\Delta)})\cap \sigma(\Delta)$, that is 
$X_{\sigma(\Delta)}$ is normal.
\end{proof}

\begin{cor}\label{ncn}
Suppose the non-normal locus $C$ of $X$ is not empty. Then either $C$ is irreducible and normal, or 
$C=\cup_iC_i$ is reducible and its non-normal locus is $\cup_{i\ne j}C_i\cap C_j$.
\end{cor}

\begin{proof} Suppose $C$ is irreducible. Then $C$ is the core of $X$, hence normal by 
Proposition~\ref{coredef}. Suppose $C$ is reducible, with irreducible components
$C_i$. If $C_i\ne C_j$, the intersection $C_i\cap C_j$ is contained in the non-normal locus
of $C$. Therefore the non-normal locus of $C$ contains $\cup_{i\ne j}C_i\cap C_j$.
On the other hand, $C_i\setminus \cup_{j\ne i}C_j$ is normal (after localization, we obtain 
$C=C_i$ irreducible, hence normal by the above argument). Therefore the non-normal locus 
of $C$ is $\cup_{i\ne j}C_i\cap C_j$.
\end{proof}


\section{Weakly normal log pairs}


Let $X/k$ be an algebraic variety, weakly normal and $S_2$.
Let $\pi\colon \bar{X}\to X$ be the normalization. The ideal sheaf $\{f\in \cO_X; 
f\cdot \pi_*\cO_{\bar{X}}\subseteq \cO_X\}$ is also an ideal sheaf on $\bar{X}$,
and cuts out the {\em conductor subschemes} $C\subset X$ and $\bar{C}\subset \bar{X}$.
We obtain a cartesian diagram
\[ 
\xymatrix{
  \bar{X} \ar[d]_\pi & \bar{C} \ar@{_{(}->}[l]  \ar[d]^\pi    \\
    X       & C  \ar@{_{(}->}[l]  
} \]
Each irreducible component of $X$ has the same dimension, equal to $d=\dim X$.
Both $C\subset X$ and $\bar{C}\subset \bar{X}$ are reduced subschemes, of pure
codimension one, and {\em $C$ is the non-normal locus of $X$}. The morphism $\pi\colon \bar{C}\to C$
is finite, mapping irreducible components onto irreducible components.
Denote by $Q(X)$ the $k$-algebra consisting of rational functions which are regular on 
an open dense subset of $X$. We have an isomorphism $\pi^*\colon Q(X)\isoto Q(\bar{X})$ and
a monomorphism $\pi^*\colon Q(C)\to Q(\bar{C})$. 
Let $B$ be the closure in $X$ of a $\Q$-Cartier divisor on the smooth locus of $X$,
and $\bar{B}$ the closure in $\bar{X}$. 

\begin{defn}
For $n\in \Z$, define a coherent $\cO_X$-module $\omega^{[n]}_{(X/k,B)}$ as follows:
for an open subset $U\subseteq X$, let $\Gamma(U,\omega^{[n]}_{(X/k,B)})$ be the set of 
rational $n$-differential forms $\omega\in (\wedge^d \Omega^1_{Q(X)/k})^{\otimes n}$ satisfying 
the following properties 
\begin{itemize}
\item[a)] $(\pi^*\omega)+n(\bar{C}+\bar{B})\ge 0$ on $\pi^{-1}(U)$.
\item[b)] If $P$ is an irreducible component of $C\cap U$, there exists a rational 
$n$-differential form $\eta\in (\wedge^{d-1} \Omega^1_{Q(P)/k})^{\otimes n}$ such that 
$\Res_Q \pi^*\omega=\pi^*\eta$ for every irreducible component $Q$ of $\bar{C}$ lying over $P$.
\end{itemize}
\end{defn}

We have natural multiplication maps 
$
\omega^{[m]}_{(X/k,B)} \otimes_{\cO_X} \omega^{[n]}_{(X/k,B)} \to \omega^{[m+n]}_{(X/k,B)} \ (m,n\in \Z).
$
By seminormality, $\omega^{[0]}_{(X/k,B)}=\cO_X$.

\begin{lem}\label{c1f} 
Suppose $rB$ has integer coefficients in a neighborhood of a codimension one point $P\in X$.
Then in a neighborhood of $P$, $\omega^{[r]}_{(X/k,B)}$ is invertible and $(\omega^{[r]}_{(X/k,B)})^{\otimes n}
\isoto \omega^{[rn]}_{(X/k,B)}$ for all $n\in \Z$.
\end{lem}

\begin{proof} Suppose $X/k$ is smooth at $P$. 
Let $t$ be a local parameter at $P$ and $\omega_0$ a local generator of $(\wedge^d\Omega^1_{X/k})_P$,
and $b=\mult_P(B)$.
If $n\in \Z$, then $t^{-\lfloor nb\rfloor}\omega^{\otimes n}$ is a local generator of $\omega^{[n]}_{(X/k,B)}$.
The claim follows.

Suppose $X$ is singular at $P$. Let $(Q_j)_j$ be the finitely many prime divisors of $\bar{X}$ lying over $P$. 
We may localize at $P$ and suppose $C=P$, $B=0$, and $\bar{C}=\sum_j Q_j$. For every $j$,
we have finite surjective maps $\pi|_{Q_j}\colon Q_j\to P$.
By weak approximation~\cite[Chapter 10, Theorem 18]{ZS60}, there exists an invertible rational function 
$t_1\in Q(\bar{X})$ which induces a local parameter at $Q_j$, for every $j$. Let $f_2,\ldots,f_d$ be a separating
transcendence basis of $k(C)/k$. For every $2\le i\le d$, there exists $t_i\in Q(\bar{X})$, regular
at each $Q_j$, such that $t_i|_{Q_j}=(\pi|_{Q_j})^*(f_i)$ for every $j$. Set 
$$
\omega=\frac{dt_1}{t_1}\wedge dt_2\wedge\cdots \wedge dt_d\in \wedge^d\Omega^1_{Q(\bar{X})/k}.
$$ 
Since $t_1\omega$ is regular, we have $(\omega)+\bar{C}\ge 0$. On the other hand,
$$
\Res_{Q_j} \omega =(\pi|_{Q_j})^*(df_2\wedge \cdots \wedge df_d ).
$$
The right hand side is non-zero, hence $(\omega)+\bar{C}=0$. Property b) is also satisfied, so 
$\omega$ belongs to $\omega^{[1]}_{(X/k,B)}$. We claim that 
$\omega^{\otimes n}$ is a local generator of $\omega^{[n]}_{(X/k,B)}$ at $P$, for all $n\in \Z$.
Indeed, let $\omega'$ be a local section of $\omega^{[n]}_{(X/k,B)}$ at $P$. There exists a regular function
$f$ on $\bar{X}$ such that $\pi^*\omega'=f\cdot \omega^{\otimes n}$. By assumption, there exists a rational 
$n$-differential $\eta\in (\wedge^{d-1} \Omega^1_{Q(P)/k})^{\otimes n}$ such that 
$\Res_{Q_j}\pi^*\omega'=(\pi|_{Q_j})^*(\eta)$ for every irreducible component $Q_j$. 
Let $\eta=h\cdot (df_2 \wedge \cdots\wedge df_d)^{\otimes n}$. We obtain
$f|_{Q_j}=(\pi|_{Q_j})^*(h)$ for all $Q_j$. By seminormality, this means that $f\in \cO_{X,P}$.
\end{proof}

\begin{cor} Let $r\ge 1$ such that $rB$ has integer coefficients.
There exists an open subset $U\subseteq X$ such that $\codim(X\setminus U,X)\ge 2$,
$\omega^{[r]}_{(X/k,B)}|_U$ is invertible and $(\omega^{[r]}_{(X/k,B)}|_U)^{\otimes n}
\isoto \omega^{[rn]}_{(X/k,B)}|_U$ for all $n\in \Z$.
\end{cor}

\begin{lem}\label{lg} 
Let $U\subseteq X$ be an open subset and $\omega\in (\wedge^d\Omega^1_{k(X)/k})^{\otimes n} \setminus 0$. 
Then $1\mapsto \omega$ induces an isomorphism $\cO_U\isoto \omega^{[n]}_{(X/k,B)}|_U$ if and
only if $(\pi^*\omega)+\lfloor n(\bar{C}+\bar{B})\rfloor=0$ on $\bar{U}=\pi^{-1}(U)$ and 
$\Res_{\bar{C}\cap \bar{U}}(\pi^*\omega)\in (\wedge^{d-1}\Omega^1_{Q(\bar{C}\cap \bar{U} )/k})^{\otimes n}$ belongs to the image of 
$\pi^*\colon (\wedge^{d-1}\Omega^1_{Q(C\cap U)/k})^{\otimes n}\to 
(\wedge^{d-1}\Omega^1_{Q(\bar{C}\cap \bar{U} )/k})^{\otimes n}$.
\end{lem}

\begin{proof} The homomorphism is well defined if and only if $(\pi^*\omega)+\lfloor n(\bar{C}+\bar{B})\rfloor\ge 0$ 
on $\bar{U}=\pi^{-1}(U)$ and $\Res_{\bar{C}\cap \bar{U}}(\pi^*\omega)=\pi^*\eta$ for 
$\eta\in (\wedge^{d-1}\Omega^1_{Q(C\cap U)/k})^{\otimes n}$. Suppose the homomorphism is an 
isomorphism. It follows that $(\pi^*\omega)+\lfloor n(\bar{C}+\bar{B})\rfloor=0$ on $\bar{U}$, since in the proof of
Lemma~\ref{c1f} we constructed local generators with this property near each codimension one point of $X$.
It follows that $\eta$ is non-zero on each irreducible component of $C\cap U$.

Conversely, let $V\subseteq U$ be an open subset and $\omega'\in \Gamma(V,\omega^{[n]}_{(X/k,B)})$. 
Then $\omega'=f\omega$, with $f\in \Gamma(\pi^{-1}(V),\cO_{\bar{X}})$. By definition,
$\Res_{\bar{C}}(\omega')=\pi^*\eta'$. Since $\eta$ is non-zero on each irreducible component of $C$,
$h=\eta'/\eta$ is a well defined rational function on $C\cap V$. Comparing residues, we obtain that
for every irreducible component $P$ of $V\cap C$, for every prime divisor $Q$ lying over $P$, we have
$f|_Q=\pi^*h$. Since $X$ is seminormal and $S_2$, this means that $f\in \Gamma(V,\cO_X)$. Therefore $\omega$
generates $\omega^{[n]}_{(X/k,B)}$ on $U$.
\end{proof}

\begin{cor} Suppose $rB$ has integer coefficients and $\omega^{[r]}_{(X/k,B)}$ is an invertible $\cO_X$-module. 
Then:
\begin{itemize}
\item[a)] 
$
\omega^{[r]}_{(X/k,B)} \otimes_{\cO_X} \omega^{[n]}_{(X/k,B)} \to \omega^{[r+n]}_{(X/k,B)}
$
is an isomorphism, for every $n\in \Z$. In particular, the graded $\cO_X$-algebra
$\oplus_{n\in \N}\omega^{[n]}_{(X/k,B)}$ is finitely generated, and 
$(\omega^{[r]}_{(X/k,B)})^{\otimes n}\isoto \omega^{[rn]}_{(X/k,B)}$ 
for every $n\in \Z$.
\item[b)] $\pi^* \omega^{[r]}_{(X/k,B)}= \omega^{[r]}_{(\bar{X}/k,\bar{C}+\bar{B} )}$.
\end{itemize}
\end{cor}

\begin{proof}
a) Similar to normal case, using moreover the fact that residues commute with multiplication of 
pluri-differential forms.

b) It follows from Lemma~\ref{lg}.
\end{proof}

We may restate property b) as saying that the normalization $(\bar{X}/k,\bar{C}+\bar{B} )\to (X/k,B)$ is log crepant.
Note that $\bar{X}$ is normal, but possibly disconnected. 

\begin{defn}
A {\em weakly normal log pair} $(X/k,B)$ consists of an algebraic variety $X/k$,
weakly normal and $S_2$, the (formal) closure $B$ of a $\Q$-Weil divisor on
the smooth locus of $X/k$, subject to the following property: there exists an integer $r\ge 1$
such that $rB$ has integer coefficients and the $\cO_X$-module $\omega^{[r]}_{(X/k,B)}$ is invertible.

If $B$ is effective, we call $(X/k,B)$ a {\em weakly normal log variety}.
\end{defn}

If $X$ is normal, these notions coincide with {\em log pairs} and {\em log varieties}.

Let $D$ be a $\Q$-Cartier divisor on $X$ supported by primes at which $X/k$ is smooth.
If $(X,B)$ is a weakly normal log pair, so is $(X,B+D)$.


\subsection{Weakly log canonical singularities, lc centers}

Suppose $\Char(k)=0$, or log resolutions exist (e.g. in the toric case). 
Note that any desingularization of $X$ factors through the normalization of $X$.
A {\em log resolution} $\mu\colon X'\to (X,B)$ is a composition $\mu=\pi\circ \bar{\mu}$, where
$\bar{\mu}\colon X'\to (\bar{X}/k,\bar{C}+\bar{B} )$ is a log resolution. 

We say that $(X/k,B)$ has {\em weakly log canonical (wlc) singularities} if $(\bar{X}/k,\bar{C}+\bar{B} )$
has log canonical singularities. The image $(X/k,B)_{-\infty}=\pi((\bar{X}/k,\bar{C}+\bar{B} )_{-\infty})$ 
is called the {\em non-wlc locus} of $(X/k,B)$. It is the complement of the largest open subset of $X$
where $(X/k,B)$ has weakly log canonical singularities. An {\em lc center} of $(X/k,B)$ is defined
as the $\pi$-image of an lc center of $(\bar{X}/k,\bar{C}+\bar{B})$, which is not contained in $(X/k,B)_{-\infty}$.
For example, the irreducible components of $X$ are lc centers. From the normal case, it follows 
that $(X/k,B)$ has only finitely many lc centers.

\begin{rem} 
If $ (\bar{X}/k,\bar{C}+\bar{B} )_{-\infty} = \pi^{-1}((X/k,B)_{-\infty})$, 
then $\pi$ maps lc centers onto lc centers.
\end{rem}


\subsection{Residues in codimension one lc centers, different}

Let $(X/k,B)$ be a weakly normal log pair. Suppose $X$ is not normal. 
Let $C$ be the non-normal locus of $X$, and $j\colon C^n\to C$ the normalization. 
We obtain a commutative diagram
\[ 
\xymatrix{
  \bar{X} \ar[d]_\pi & \bar{C} \ar@{_{(}->}[l]  \ar[d]_\pi &   \bar{C}^n  \ar[l]_i  \ar[d]^g    \\
    X       & C  \ar@{_{(}->}[l]  &   C^n  \ar[l]_j  
} \]
Pick $l\in \Z$ such that $lB$ has integer coefficients and $\omega^{[l]}_{(X/k,B)}$ is invertible.
We will naturally define a structure of log pair $(C^n/k,B_{C^n})$ and isomorphisms
$$
\Res^{[l]} \colon \omega^{[l]}_{(X/k,B)}|_{C^n}\isoto \omega^{[l]}_{(C^n/k,B_{C^n})}.
$$
Indeed, suppose moreover that $\cO_X\isoto \omega^{[l]}_{(X/k,B)}$. 
Let $\omega$ be the corresponding global generator. We have $(\pi^*\omega)+l(\bar{C}+\bar{B})=0$, and 
$\Res^{[l]}_{\bar{C}^n} \pi^*\omega =g^*\eta$ for some $\eta\in (\wedge^{d-1}\Omega^1_{Q(C)/k})^{\otimes l}$. 
It follows that $\eta$ is non-zero on each component of $C$. 

Note that $\eta=\eta(\omega)$ is uniquely determined by $\omega$.
If $\omega'$ is another global generator, it follows that $\omega'=f\omega$ for a global unit
$f\in \Gamma(X,\cO_X^\times)$. Therefore $\eta(\omega')=(f|_C)\cdot \eta(\omega)$ and 
$f|_C$ is a global unit on $C$. Therefore the $\Q$-Weil divisor on $C^n$
$$
B_{C^n}=-\frac{1}{l}(\eta)
$$
does not depend on the choice of a generator $\omega$. It follows that the definition of $B_{C^n}$ 
makes sense globally if $\omega^{[l]}_{(X/k,B)}$ is just locally free, since we can patch local
trivializations. The definition does not depend on the choice of $l$ either. 

Denote by $i'\colon \bar{C}^n\to \bar{X}$ and $j'\colon C^n\to X$ the induced morphisms. Let $B_{\bar{C}^n}$
be the different of $(\bar{X},\bar{C}+\bar{B})$ on (each connected component of) $\bar{C}^n$.
We have isomorphisms
$$
\pi^*\omega^{[l]}_{(X/k,B)} \isoto \omega^{[l]}_{(\bar{X}/k,\bar{C}+\bar{B})},
\Res^{[l]} \colon {i'}^*\omega^{[l]}_{(\bar{X}/k,\bar{C}+\bar{B} )} \isoto \omega^{[l]}_{(\bar{C}^n/k,B_{\bar{C}^n})},
g^*\omega^{[l]}_{(C^n/k,B_{C^n})}\isoto \omega^{[l]}_{(\bar{C}^n/k,B_{\bar{C}^n})}.
$$
In particular, we obtain an isomorphism
$
 {j'}^*\omega^{[l]}_{(X/k,B)}\isoto \omega^{[l]}_{(C^n/k,B_{C^n})}.
$
We may say that in the following commutative diagram, all maps are log crepant:
\[ 
\xymatrix{
  (\bar{X},\bar{C}+\bar{B}) \ar[d]_\pi &   (\bar{C}^n,B_{\bar{C}^n})  \ar[l]_{i'}  \ar[d]^g    \\
   (X,B)        &   (C^n,B_{C^n})  \ar[l]_{j'}  
} \]

\begin{lem}[Inversion of adjunction] Suppose $\Char(k)=0$ and $B\ge 0$.
Then $(X,B)$ has wlc singularities near $C$ if and only if
$(C^n,B_{C^n})$ has lc singularities.
\end{lem}

\begin{proof} We have $\bar{C}=\pi^{-1}(C)$. Therefore $(X,B)$ has wlc singularities near $C$ if 
and only if $(\bar{X},\bar{C}+\bar{B})$ has lc singularities near $\bar{C}$. By~\cite{Kwk07}, this 
holds if and only if $(\bar{C}^n,B_{\bar{C}^n})$ has lc singularities. Since $g$ is a finite log crepant 
morphism, the latter holds if and only if $(C^n,B_{C^n})$ has lc singularities.
\end{proof}

If $B$ is effective, then $B_{C^n}$ is effective.

Let $E$ be an lc center of $(X/k,B)$ of codimension one. Let $E^n\to E$ be the normalization.
Then there exists a log pair structure $(E^n,B_{E^n})$ on the normalization of $E$, together
with residue isomorphisms 
$
\Res_E^{[r]}\colon \omega^{[r]}_{(X/k,B)}|_{E^n}\isoto \omega^{[r]}_{(E^n,B_{E^n})},
$
for every $r\in \Z$ such that $rB$ has integer coefficients and $\omega^{[r]}_{(X/k,B)}$ is 
invertible.
Indeed, if $X$ is normal at $E$, we have the usual codimension one residue. Else, 
$E$ is an irreducible component of $C$ and $E^n$ is an irreducible component of $C^n$,
and the residue isomorphism and different was constructed above.


\subsection{Semi-log canonical singularities}

Suppose $\Char(k)=0$.
We show that semi-log canonical pairs are exactly the weakly normal log varieties which have
wlc singularities and are Gorenstein in codimension one.
Recall~\cite[Definition-Lemma 5.10]{Kbook} that a {\em semi-log canonical pair} $(X/k,B)$ consists of an algebraic variety
$X/k$ which is $S_2$ and has at most nodal singularities in codimension one, and 
an effective $\Q$-Weil divisor $B$ on $X$, supported by nonsingular codimension one points of $X$, such that 
the following properties hold:
\begin{itemize}
\item[1)] There exists $r\ge 1$ such that $rB$ has integer coefficients and the $\cO_X$-module 
$\omega^{[r]}_X(rB)$ is invertible. This sheaf is constructed as follows: there exists an open 
subset $w\colon U\subseteq X$ such that $\codim(X\setminus U\subset X)\ge 2$, $U$ has Gorenstein 
singularities and $rB|_U$ is Cartier. Let $\omega_U$ be a dualizing sheaf on $U$, which is invertible. 
Then $\omega^{[r]}_X(rB)=w_*(\omega_U^{\otimes r}\otimes \cO_U(rB|_U))$.

If we consider the normalization of $X$ and the conductor subschemes
\[ 
\xymatrix{
  \bar{X} \ar[d]_\pi & \bar{C} \ar@{_{(}->}[l]  \ar[d]^\pi    \\
    X       & C  \ar@{_{(}->}[l]  
} \]
we obtain $\pi^*\omega^{[r]}_X(rB)\isoto \omega^{[r]}_{\bar{X}}(r\bar{C}+r\bar{B})$, where 
$\bar{B}=\pi^*B$ is the pullback as a $\Q$-Weil divisor.
\item[2)] $(\bar{X},\bar{C}+\bar{B})$ is a log variety (possibly disconnected) with at most log canonical
singularities. 
\end{itemize}

On the normal variety $\bar{X}$, we have $\omega^{[r]}_{\bar{X}}(r\bar{C}+r\bar{B})= \omega^{[r]}_{(\bar{X}/k,\bar{C}+\bar{B})}$.
The normalizations of $\bar{C}$ and $C$ induce a commutative diagram
\[ 
\xymatrix{
  \bar{X} \ar[d]_\pi & \bar{C} \ar@{_{(}->}[l]  \ar[d] &   \bar{C}^n  \ar[l]  \ar[d]^g    \\
    X       & C  \ar@{_{(}->}[l]  &   C^n  \ar[l]
} \]
The assumption that the non-normal codimension one singularities of $X$ are nodal means that 
$g$ is $2\colon 1$. Equivalently, $g$ is the quotient of $\bar{C}^n$ by an involution 
$\tau\colon \bar{C}^n\to \bar{C}^n$. If we further assume $2\mid r$, we obtain by~\cite[Proposition 5.8]{Kbook}
that $\omega^{[r]}_X(rB)$ consists of the section $\omega$ of $\omega^{[r]}_{\bar{X}}(r\bar{C}+r\bar{B})$
whose residue $\omega'$ on $\bar{C}^n$ is $\tau$-invariant, which is equivalent to $\omega'$ being pulled back
from $C^n$. We obtain
$$
\omega^{[r]}_X(rB)= \omega^{[r]}_{(X/k,B)} \ (2\mid r).
$$
Since nodal singularities are weakly normal and Gorenstein, we conclude that $(X/k,B)$ is a weakly normal log variety 
with wlc singularities, which is Gorenstein in codimension one. Moreover, 
$\omega^{[r]}_X(rB)= \omega^{[r]}_{(X/k,B)}$ if $2\mid r$.

Conversely, let $(X/k,B)$ be a weakly normal log variety with wlc singularities, which is Gorenstein 
in codimension one. Among weakly normal points of codimension one, only smooth and nodal ones are Gorenstein.
It follows that $(X/k,B)$ is a semi-log canonical pair, and 
$
\omega^{[n]}_X(nB)= \omega^{[n]}_{(X/k,B)} 
$
for every $n\in 2\Z$.

Note that for a weakly normal log variety with wlc singularities $(X/k,B)$, the following are equivalent:
\begin{itemize}
\item $(X/k,B)$ is a semi-log canonical pair.
\item $X$ is Gorenstein in codimension one.
\item If $X$ is not normal, the induced morphism $g\colon \bar{C}^n \to C^n$ is $2:1$.
\end{itemize}


\section{Toric weakly normal log pairs}



\subsection{Irreducible case}


Let $X=\Spec k[S]$ be weakly normal and $S_2$. It is an equivariant embedding of
the torus $T=\Spec k[M]$, where $M=S-S$ (see~\cite[Section 2]{TFR1}). 
Let $\pi\colon \bar{X}\to X$ be the 
normalization, with induced conductor subschemes $\pi\colon \bar{C}\to C$.
Let $\{\tau_i\}_i$ be the codimension one faces of $\sigma_S$. Then 
$E_i=\Spec k[S_{\tau_i}]$ are the invariant codimension one subvarieties of $X$,
and $\bar{E_i}=\Spec k[M\cap \tau_i]$ are the invariant codimension one 
subvarieties of $\bar{X}$. Each $\bar{E}_i$ is normal, and the following diagram is cartesian: 
\[ 
\xymatrix{
  \bar{X} \ar[d]_\pi & \bar{E}_i \ar@{_{(}->}[l]  \ar[d]^{\pi_i}    \\
    X       & E_i  \ar@{_{(}->}[l]  
} \]
Each morphism $\pi_i \colon \bar{E}_i\to E_i$ is finite surjective of degree $d_i$,
the incidence number of $E_i\subset X$.

Let $X_{\sigma(\Delta)}$ be the core of $X$. We have $\sigma=\sigma_S$ if $X$ is normal, and 
$\sigma(\Delta)=\cap_{d_i>1}\tau_i$ otherwise. Denote $\Sigma_{\bar{X}}=\bar{X}\setminus T=\sum_i \bar{E}_i$. 

Let $B=\sum_i b_i E_i$ be a $\Q$-Weil divisor on $X$ supported by prime
divisors in which $X/k$ is smooth. Note that $X/k$ is smooth at $E_i$ if and only if 
$E_i\not\subset C$, if and only if $d_i=1$.

\begin{lem}\label{incr} Let $n\in \Z$. The following properties are equivalent:
\begin{itemize}
\item[a)] $\omega^{[n]}_{(X/k,B)}$ is invertible at some point $x$, which belongs to the closed orbit of $X$.
\item[b)] $\cO_X\simeq \omega^{[n]}_{(X/k,B)}$.
\item[c)] There exists $m\in S_{\sigma(\Delta)}-S_{\sigma(\Delta)}$ such that 
$(\chi^m)+\lfloor n(-\Sigma_{\bar{X}}+\bar{C}+\bar{B}) \rfloor=0$ on $\bar{X}$.
\end{itemize}
\end{lem}

\begin{proof} $a) \Longrightarrow c)$ The torus $T$ acts on $\omega^{[n]}_{(X/k,B)}$, hence on 
$\Gamma(X,\omega^{[n]}_{(X/k,B)})$. By the complete reducibility theorem, the space of global
sections decomposes into one-dimensional invariant subspaces. Therefore the space of global sections is generated
by semi-invariant pluri-differential forms. Since $X$ is affine, $\omega^{[n]}_{(X/k,B)}$ is generated 
by its global sections. Suppose $\omega^{[n]}_{(X/k,B)}$ is invertible at $x$. Then there exists a semi-invariant
global section $\omega\in \Gamma(X,\omega^{[n]}_{(X/k,B)})$ which induces a local trivialization near $x$.

Let $\bar{x}$ be a point of $\bar{X}$ lying over $x$. Then $\pi^*\omega$ is a local trivialization for 
$\omega^{[n]}_{(\bar{X}/k,\bar{C}+\bar{B})}$ near the point $\bar{x}$, which belongs to the closed orbit 
of $\bar{X}$. By Lemma~\ref{NI}, there exists $m\in M$ such that 
$(\chi^m)+\lfloor n(-\Sigma_{\bar{X}}+\bar{C}+\bar{B}) \rfloor=0$ on $\bar{X}$. Then 
$\chi^m\omega_M^{\otimes n}$ becomes a nowhere zero global section of 
$\omega^{[n]}_{(\bar{X}/k,\bar{C}+\bar{B})}$, where $\omega_M$ is the volume form on $T$ induced by an
orientation of $M$.

Now $\pi^*\omega=f \cdot \chi^m\omega_\cB^{\otimes n}$, for some $f\in \Gamma(\bar{X},\cO_{\bar{X}})$ which 
is a unit at $\bar{x}$. Since $\omega$ is semi-invariant, so is $f$. Therefore $f=c\chi^u$ for some $c\in k^\times$
and $u\in M$. Since $f$ is a unit at $\bar{x}$, it is a global unit, that is $u\in \bar{S}\cap (-\bar{S})$. 
Replacing $\omega$ by $\omega/c$ and $m$ by $m+u$, we may suppose 
$$
\pi^*\omega=\chi^m\omega_M^{\otimes n}.
$$
Let $E_i\subseteq C$ be an irreducible component. 
The identity $(\chi^m)+\lfloor n(-\Sigma_{\bar{X}}+\bar{C}+\bar{B}) \rfloor=0$
at $\bar{E}_i$ is equivalent to $m\in M\cap \tau_i-M\cap \tau_i$. We compute $\chi^m|_{\bar{E}_i}=\chi^m$ and
$$
\Res^{[n]}_{\bar{E}_i} \pi^*\omega= \chi^m \cdot (\Res_{\bar{E}_i} \omega_M)^{\otimes n}.
$$
Let $\omega_i$ be a volume form on the torus inside $E_i$ induced by an orientation of $S_{\tau_i}-S_{\tau_i}$,
let $\bar{\omega}_i$ be a volume form on the torus inside $\bar{E}_i$ induced by an orientation
of $M\cap \tau_i-M\cap \tau_i$. Then $\pi^*\omega_i=(\pm d_i)\cdot \bar{\omega}_i$ and 
$\Res_{\bar{E}_i} \omega_M=(\pm 1) \cdot \bar{\omega}_i$. Since $X/k$ is weakly normal, $\Char(k) \nmid d_i$. Thus
$
\Res_{\bar{E}_i} \omega_M=\pi_i^*( (\epsilon_i d_i)^{-1}\omega_i)
$
for some $\epsilon_i=\pm 1$. Therefore $\Res_{\bar{E}_i} \pi^*\omega$ is pulled back from the generic point of 
$E_i$ if and only if so is $\chi^m\in k(\bar{E}_i)$, which is equivalent to $m\in S_{\tau_i}-S_{\tau_i}$.
In particular, $m$ belongs to $M\cap \cap_{d_i>1}(S_{\tau_i}-S_{\tau_i})$, which is $S_{\sigma(\Delta)}-S_{\sigma(\Delta)}$
by Proposition~\ref{coredef}.

$c) \Longrightarrow b)$ $\chi^m\omega_M^{\otimes n}$ becomes a nowhere zero global
section $\omega\in \Gamma(X,\omega^{[n]}_{(X/k,B)})$, with
$$
\Res^{[n]}_{\bar{E}_i} \pi^*\omega=\pi_i^*( (\epsilon_id_i)^{-n}\chi^m\omega_i^{\otimes n}) .
$$

$b) \Longrightarrow a)$ is clear.
\end{proof}

\begin{prop}\label{IC} 
$(X/k,B)$ is a weakly normal log pair if and only if $(\bar{X}/k,\bar{C}+\bar{B})$ is a log pair.
Moreover:
\begin{itemize} 
\item $B$ is effective if and only if $\bar{C}+\bar{B}$ is effective.
\item $(X/k,B)$ has wlc singularities if and only if $(\bar{X}/k,\bar{C}+\bar{B})$ has lc singularities, if and only if 
the coefficients of $B$ are at most $1$.
\item $(X/k,B)$ has slc singularities if and only if $d_i\mid 2$ for all $i$.
\end{itemize}
\end{prop}

\begin{proof} Denote $d=\lcm_i d_i$. Pick $r\ge 1$ such that $rB$ has integer coefficients.
If $\omega^{[r]}_{(X/k,B)}$ is invertible, so is $\pi^*\omega^{[r]}_{(X/k,B)}=\omega^{[r]}_{(\bar{X}/k,\bar{C}+\bar{B})}$.
Conversely, the sheaf $\omega^{[r]}_{(\bar{X}/k,\bar{C}+\bar{B})}$ is invertible if and only if there exists 
$m\in M$ such that $(\chi^m)+r(-\Sigma_{\bar{X}}+\bar{C}+\bar{B})=0$ on $\bar{X}$. 
Let $E_i\subset \bar{C}$. Since $m\in \bar{S}_{\tau_i}-\bar{S}_{\tau_i}$, 
$dm\in S_{\tau_i}- S_{\tau_i}$. Since $(\chi^{dm})+dr(-\Sigma_{\bar{X}}+\bar{C}+\bar{B})=0$ on $\bar{X}$,
$\omega^{[dr]}_{(X/k,B)}$ is invertible by Lemma~\ref{incr}.

Note that $\psi=\frac{1}{r}m\in (S_{\sigma(\Delta)}-S_{\sigma(\Delta)})_\Q$ is a log discrepancy function of the toric 
log pair $(\bar{X}/k,\bar{C}+\bar{B})$.
We deduce that $(X/k,B)$ has wlc singularities if and only if $(\bar{X},\bar{C}+\bar{B})$ has lc singularities,
if and only if the coefficients of $B$ are at most $1$, if and only if $\psi\in \sigma_S$.
\end{proof}

A log discrepancy function $\psi$ is unique modulo the vector space $\sigma_S\cap (-\sigma_S)$, the largest 
vector space contained in $\sigma_S$, or equivalently, the smallest face of $\sigma_S$.
We actually have $\psi\in \sigma(\Delta)$.

\begin{lem} Suppose $(X/k,B)$ is a weakly normal log pair, with log discrepancy function $\psi$.
\begin{itemize} 
\item[1)] $(X/k,B)_{-\infty}=\cup_{b_i>1}E_i$ and $(\bar{X}/k,\bar{C}+\bar{B})_{-\infty}=\cup_{b_i>1}\bar{E}_i=\pi^{-1}((X,B)_{-\infty})$.
\item[2)] The lc centers of $(X/k,B)$ are $X_\sigma$, where $\psi\in \sigma\prec \sigma_S$ and 
$\sigma\not\subset \tau_i$ if $b_i>1$.
\item[3)] The correspondence $Z\mapsto \pi^{-1}(Z)$ is one to one between lc 
centers of $(X/k,B)$ and lc centers of $(\bar{X}/k,\bar{C}+\bar{B})$.
\end{itemize}
\end{lem}

Suppose $(X/k,B)$ is wlc, with log discrepancy function $\psi\in \sigma_S$.
The lc centers of $(X/k,B)$ are $X_\sigma$, where $\psi\in \sigma\prec \sigma_S$. For $\sigma=\sigma_S$,
we obtain the lc center $X$, for $\sigma\ne \sigma_S$ we obtain lc centers defined by toric valuations.
Any union of lc centers is weakly normal. The intersection of two lc centers is again an lc center. 
With respect to inclusion, there exists a unique minimal lc center, namely $X_{\sigma(\psi)}$ for 
$\sigma(\psi)=\cap_{\psi\in \sigma\prec \sigma_S}\sigma$ (the unique face of $\sigma_S$ which contains 
$\psi$ in its relative interior). 
Note that $X$ is the unique lc center of $(X/k,B)$ if and only if $X$ is normal and the coefficients of $B$
are strictly less than $1$.

\begin{lem}\label{mlcn}
Suppose $(X/k,B)$ is wlc. Then the minimal lc center of $(X/k,B)$ is normal. 
\end{lem}

\begin{proof} Let $X_{\sigma(\Delta)}$ be the core of $X$. It is an intersection of lc centers of 
$(X/k,B)$, hence an lc center itself. Equivalently, $\sigma(\psi)\prec \sigma(\Delta)$ and 
the minimal lc center $X_{\sigma(\psi)}$ is an invariant closed subvariety of $X_{\sigma(\Delta)}$

By Proposition~\ref{coredef}, the core is normal. Then so is each invariant closed irreducible 
subvariety of the core. Therefore $X_{\sigma(\psi)}$ is normal.
\end{proof}

\begin{exmp}
Let $X/k$ be an irreducible affine toric variety, weakly normal and $S_2$. 
Let $\Sigma$ be the sum of codimension one subvarieties at which $X/k$ is smooth. 
Then $\cO_X\isoto \omega_{(X/k,\Sigma)}^{[1]}$ and $(X/k,\Sigma)$ is a weakly normal log variety
with wlc singularities.

Indeed, let $\omega$ be the volume form on $T=\Spec k[M]$ induced by an orientation of $M$.
Then $(\omega)+\Sigma_{\bar{X}}=0$ on $\bar{X}$. Its residues descend by weak normality
(cf. the proof of Lemma~\ref{incr}), so $\omega$ becomes a nowhere zero global section of 
$\omega_{(X/k,\Sigma)}^{[1]}$. Since $\bar{C}+\bar{\Sigma}=\Sigma_{\bar{X}}$ and 
$(\bar{X},\Sigma_{\bar{X}})$ has lc singularities, the claim holds.
\end{exmp}

The lc centers of $(X/k,B)$ of codimension one are the invariant primes $E_i$ such that
either $\mult_{E_i}B=1$, or $E_i$ is an irreducible component of $C$. The normalization of 
$E_i$ is $E^n_i=\Spec k[(S_{\tau_i}-S_{\tau_i})\cap \tau_i]$, the different $B_{E^n_i}$
is induced by the log discrepancy function $\psi$ of $(X/k,B)$, and the residue of 
$\chi^{r\psi}\omega^{\otimes r}$ is $(\epsilon_i d_i)^{-1}\chi^{r\psi}\omega_{\cB_i}^{\otimes r}$.


\subsection{Reducible case}


Let $X=\Spec k[\cM]$ be weakly normal and $S_2$. Let $\{F\}$ and $\{\tau_i\}$
be the facets and codimension one faces of $\Delta$, respectively. The 
normalization $\pi\colon \bar{X}\to X$ is $\sqcup_F \bar{X}_F\to \cup_FX_F$,
where $\bar{X}_F=\Spec k[\bar{S}_F]$ and $\bar{S}_F=(S_F-S_F)\cap F$.
The invariant codimension one subvarieties of $X$ are $E_i=\Spec k[S_{\tau_i}]$
(either irreducible components of $C$, or invariant prime divisors at which $X/k$ 
is smooth). Note that
$
\pi^{-1}(E_i)=\sqcup_F (\bar{X}_F)_{\tau_i\cap F}
$
may have components of different dimension.
The primes of $\bar{X}$ over $E_i$ are $\bar{E}_{i,F}=(\bar{X}_F)_{\tau_i}$, 
one for each facet $F$ containing $\tau_i$.
For $F\succ \tau_i$,  $\bar{E}_{i,F}=\Spec k[\bar{S}_F\cap \tau_i]$ (note $\bar{S}_F\cap \tau_i=(S_F-S_F)\cap \tau_i$),
and the morphism $\pi_{i,F}\colon \bar{E}_{i,F} \to E_i$ is finite of degree $d_{\tau_i\prec F}$,
equal to the incidence number of $E_i\subset X_F$. Since $X/k$ is weakly normal,
$\Char(k)\nmid d_{\tau_i\prec F}$, that is $d_{\tau_i\prec F}$ is invertible in $k^\times$.
Let $X_{\sigma(\Delta)}$ be the core of $X$.

\begin{lem}\label{crc} Let $\omega_F$ be a volume form on the torus inside $X_F$, induced
by an orientation of the lattice $S_F-S_F$. Let $\omega_i$ be a volume form
on the torus inside $E_i$, induced by an orientation of the lattice $S_{\tau_i}-S_{\tau_i}$.
For $\tau_i\prec F$, there exists $\epsilon_{\tau_i\prec F}=\pm 1$ such that
$
\pi_{i,F}^*\omega_i=\epsilon_{\tau_i\prec F} d_{\tau_i\prec F}\cdot \Res_{ \bar{E}_{i,F} }\omega_F.
$

Let $n\in \Z$. The following properties are equivalent:
\begin{itemize}
\item[a)] There exist $c_F,c_i\in k^\times$ such that $\Res^{[n]}_{\bar{E}_{i,F}}(c_F\omega_F^{\otimes n})=
\pi_{i,F}^*(c_i\omega_i^{\otimes n})$ for every $\tau_i\prec F$.
\item[b)] For every cycle $F_0,F_1,\ldots,F_l,F_{l+1}=F_0$ of facets of $\Delta$ such that 
$F_i\cap F_{i+1} \ (0\le i<l)$ has codimension one, the following identity holds in $k^\times$:
$$
(\prod_{i=0}^l\frac{\epsilon_{F_i\cap F_{i+1}\prec F_{i+1}}  d_{F_i\cap F_{i+1}\prec F_{i+1}}}{
\epsilon_{F_i\cap F_{i+1}\prec F_i}  d_{F_i\cap F_{i+1}\prec F_i}})^n=1.
$$
\end{itemize}
\end{lem}

\begin{proof} Denote $e_{i,F}=(\epsilon_{\tau_i\prec F}d_{\tau_i\prec F})^n$.
Property a) is equivalent to $c_F=c_i\cdot e_{i,F}$ for every $\tau_i\prec F$.

$a)\Longrightarrow b)$ Suppose a) holds. If $(F,F')$ is a pair of facets which intersect in a 
codimenson one face, then $c_F$ determines $c_{F'}$, by the formula
$$
c_{F'}=c_F\cdot \frac{e_{F\cap F'\prec F'}}{ e_{F\cap F'\prec F}}.
$$
Let $F_0,F_1,\ldots,F_l,F_{l+1}=F_0$ be a 
cycle such that $F_i\cap F_{i+1} \ (0\le i <l)$ has codimension one. Multiplying the 
above formulas for each pair $(F_i,F_{i+1}) \ (0\le i<l)$, and factoring out the nonzero 
constants $c_{F_i}$, we obtain 
$$
\prod_{i=0}^l\frac{e_{F_i\cap F_{i+1}\prec F_{i+1}}}{e_{F_i\cap F_{i+1}\prec F_i}} =1.
$$

$b)\Longrightarrow a)$ Fix a facet $F_0$, set $c_{F_0}=1$.
Since $\Delta$ is $1$-connected, each facet $F$ is the end of a chain of facets 
$F_0,F_1,\ldots,F_l=F$ such that $F_i\cap F_{i+1}$ has codimension one for every $0\le i <l$. 
Define 
$$
c_F=\prod_{0\le i<l} \frac{e_{F_i\cap F_{i+1}\prec F_{i+1}}}{e_{F_i\cap F_{i+1}\prec F_i}}\in k^\times.
$$
The definition is independent of the choice of the chain reaching $F$, by the cycle condition b)
applied to the concatenation of two chains.
For each $\tau_i$, choose a facet $F$ containing it, and define 
$$
c_i=\frac{c_F}{e_{i,F}}.
$$
The definition is independent of the choice of $F$. Indeed, if $F,F'$ are two facets which contain
$\tau_i$, then $\tau_i=F\cap F'$. Forming a cycle with a chain from $F_0$ to $F$, followed by
$F'$, and by the reverse of a chain from $F_0$ to $F'$, we obtain from b) that 
$$
\frac{c_F}{e_{i,F}}=\frac{c_{F'}}{e_{i,F'}}.
$$
Property a) holds by construction.
\end{proof}

We say that $X/k$ is {\em $n$-orientable} if the equivalent properties of Lemma~\ref{crc} hold.
If $n$ is even, this property is independent on the choice of orientations, and becomes
$$
(\prod_{i=0}^l\frac{ d_{F_i\cap F_{i+1}\prec F_{i+1}}}{d_{F_i\cap F_{i+1}\prec F_i}})^n=1\in k^\times.
$$
We say that $X/k$ is {\em $\Q$-orientable} if it is $n$-orientable for some $n\ge 1$.

\begin{lem}\label{crcrem} Suppose $d_{\tau_i\prec F}$ does not depend on $F$.
Then $X/k$ is $n$-orientable, for every $n\in 2\Z$. In particular, $X$ is $\Q$-orientable.
\end{lem}

\begin{proof}
Since $\frac{d_{F_i\cap F_{i+1}\prec F_{i+1}}}{d_{F_i\cap F_{i+1}\prec F_i}}=1$ in this case.
\end{proof}

\begin{exmp}\label{2ex} 
Some examples where the incidence numbers $d_{\tau_i\prec F}$ do not depend on $F$ are:
\begin{itemize}
\item[1)] $X$ is irreducible.
\item[2)] $X$ has normal irreducible components (equivalent to $X_\sigma$ normal for every
$\sigma\in \Delta$). Then $d_{\tau_i\prec F}=1$ for all $\tau_i\prec F$.
\item[3)] $X$ is nodal in codimension one. Equivalently, for each codimension one face $\tau_i\in \Delta$, 
either $\tau_i$ is contained in a unique facet $F$ and $d_{\tau_i\prec F}\mid 2$, or $\tau_i$ is contained
in exactly two facets $F,F'$ and $d_{\tau_i\prec F}=d_{\tau_i\prec F'}=1$.
\end{itemize}
\end{exmp}

Let $B=\sum_i b_iE_i$ be a $\Q$-Weil divisor supported by invariant codimension one subvarieties of 
$X$ at which $X/k$ is smooth. Note that $X/k$ is smooth at $E_i$ if and only if $E_i$ is contained
in a unique irreducible component $X_F$ of $X$, and $d_{E_i\subset X_F}=1$.

\begin{lem}\label{incrred} Let $n\in \Z$. The following properties are equivalent:
\begin{itemize}
\item[a)] $\omega^{[n]}_{(X/k,B)}$ is invertible at some point $x$, which belongs to the closed orbit of $X$.
\item[b)] $\cO_X\simeq \omega^{[n]}_{(X/k,B)}$.
\item[c)] $X$ is $n$-orientable and there exists 
$m\in S_{\sigma(\Delta)}- S_{\sigma(\Delta)}$ such that
$(\chi^m)+\lfloor n(-\Sigma_{\bar{X}}+\bar{C}+\bar{B}) \rfloor=0$ on $\bar{X}_F$ for every $F$.
\end{itemize}
\end{lem}

\begin{proof} We use the definitions and notations of Lemma~\ref{crc}.

$a)\Longrightarrow c)$ As in the proof of Lemma~\ref{incr}, there exists a semi-invariant form
$\omega\in \Gamma(X,\omega^{[n]}_{(X/k,B)})$ which induces a local trivialization at $x$.
Let $F$ be a facet of $\Delta$. Let $\bar{x}_F\in \bar{X}_F$ be a point lying over $x$. Then 
$\pi^*\omega|_{\bar{X}_F}\in \Gamma(\bar{X}_F, \omega^{[n]}_{(\bar{X}/k,\bar{C}+\bar{B})})$
induces a local trivialization at $\bar{x}_F$, which belongs to the closed orbit of $\bar{X}_F$.
By Lemma~\ref{incr}, there exists $m_F\in S_F-S_F$ such that 
$(\chi^{m_F})+\lfloor n(-\Sigma_{\bar{X}}+\bar{C}+\bar{B}) \rfloor=0$ on $\bar{X}_F$,
so that $\chi^{m_F}\omega_F^{\otimes n}\in \Gamma(\bar{X}_F,
\omega^{[n]}_{(\bar{X}/k,\bar{C}+\bar{B})})$ is a nowhere zero section.
Since $\omega$ is $T$-semi-invariant, we obtain 
$
\pi^*\omega|_{\bar{X}_F}=c_F\cdot \chi^{u_F}\chi^{m_F}\omega_F^{\otimes n},
$
where $c_F\in k^\times$ and $\chi^{u_F}$ is a global unit on $\bar{X}_F$. Replacing $m_F$ by $u_F+m_F$,
we obtain 
$$
\pi^*\omega|_{\bar{X}_F}=c_F\cdot \chi^{m_F}\omega_F^{\otimes n}.
$$
By assumption, there exists $\eta_i\in \omega_{k(E_i)/k}^{\otimes n}$ such that 
for every $E_i\subset C$, and every inclusion $E_i\subset X_F$, we have 
$$
\Res^{[n]}_{\bar{E}_{i,F}} \pi^*\omega =\pi_{i,F}^* \eta_i.
$$
We have $\eta_i=f_i\omega_i^{\otimes n}$ for some $f_i\in k(E_i)^\times$. The residue formula becomes
$$
c_F\chi^{m_F}=(\epsilon_{\tau_i\prec F}d_{\tau_i\prec F})^n \pi_{i,F}^*f_i.
$$
Then $f_i$ is a unit on the torus inside $E_i$, hence $f_i=c_i\chi^{m_i}$ for some $c_i\in k^\times$
and $m_i\in S_{\tau_i}-S_{\tau_i}$. We obtain 
$$
c_F\chi^{m_F}=c_i (\epsilon_{\tau_i\prec F}d_{\tau_i\prec F})^n  \chi^{m_i}.
$$
That is $c_F=(\epsilon_{\tau_i\prec F}d_{\tau_i\prec F})^n$ and $m_F=m_i$.
Since $\Delta$ is $1$-connected, the latter means that $m_F=m_i=m$ for all $F$ and $i$, for some 
$m\in S_{\sigma(\Delta)}- S_{\sigma(\Delta)}$. The former means that $X$ is $n$-orientable. 

$c) \Longrightarrow b)$ By Lemma~\ref{crc}, there exist $c_F,c_i\in k^\times$ with
$\Res^{[n]}_{\bar{E}_{i,F}}(c_F\omega_F^{\otimes n})=\pi_{i,F}^*(c_i\omega_i^{\otimes n})$ if $
\tau_i\prec F$.
The pluridifferential forms $\{c_F\chi^m\omega_F^{\otimes n}\}_F$ on the normalization of $X$
glue to a nowhere zero global section $\omega$ of $\omega^{[n]}_{(X/k,B)}$. Moreover,
$
\Res^{[n]}_{\bar{E}_{i,F}} \pi^*\omega =\pi_{i,F}^*(c_i\omega_i^{\otimes n}).
$

$b) \Longrightarrow a)$ is clear.
\end{proof}

\begin{prop}\label{wlpC} $(X/k,B)$ is a weakly normal log pair if and only if $X$ is $\Q$-orientable, 
and the components of the normalization $(\bar{X}/k,\bar{C}+\bar{B})$ 
are toric normal log pairs with the same log discrepancy function $\psi$. Moreover,
\begin{itemize}
\item $B$ is effective if and only if $\bar{C}+\bar{B}$ is effective.
\item $(X/k,B)$ has wlc singularities if and only if $(\bar{X}/k,\bar{C}+\bar{B})$ has 
lc singularities, if and only if the coefficients of $B$ are at most $1$, if and only if 
$\psi\in \sigma(\Delta)$.
\end{itemize}
\end{prop}

\begin{proof}
Suppose $(X/k,B)$ is a weakly normal log pair. There exists an even integer $r\ge 1$ such that 
$rB$ has integer coefficients and $\omega^{[r]}_{(X/k,B)}$ is invertible. Then $\pi^*\omega^{[r]}_{(X/k,B)}=
\omega^{[r]}_{(\bar{X}/k,\bar{C}+\bar{B})}$ is invertible, hence each irreducible component of 
$(\bar{X}/k,\bar{C}+\bar{B})$ is a toric log pair. By Lemma~\ref{incrred}, $X$ is $r$-orientable
and there exists $m\in S_{\sigma(\Delta)}- S_{\sigma(\Delta)}$ such that 
$(\chi^m)+\lfloor r(-\Sigma_{\bar{X}}+\bar{C}+\bar{B})|_{\bar{X}_F} \rfloor=0$ 
for every facet $F$ of $\Delta$. Therefore $\psi=\frac{1}{r}m\in (S_{\sigma(\Delta)}- S_{\sigma(\Delta)})_\Q$ 
is a log discrepancy function for $(\bar{X}/k,\bar{C}+\bar{B})|_{\bar{X}_F}$, for each facet $F$. We call $\psi$ 
a {\em log discrepancy function} of $(X/k,B)$.

Conversely, suppose that $X$ is $\Q$-orientable, and that the irreducible (connected) components 
of the normalization $(\bar{X}/k,\bar{C}+\bar{B})$ are toric normal log pairs with the same log 
discrepancy function $\psi\in \cap_F (S_F-S_F)_\Q$. We have 
$(\chi^\psi)+(-\Sigma_{\bar{X}}+\bar{C}+\bar{B})|_{\bar{X}_F} =0$ for every facet $F$ of $\Delta$.
Choose an even integer $r\ge 1$ such that $r\psi\in \cap_F(S_F-S_F)$.
Let $\tau_i$ be a codimension one face of $\Delta$. Choose $F\succ \tau_i$.
The $\Q$-divisor $(\chi^\psi)+(-\Sigma_{\bar{X}}+\bar{C}+\bar{B})|_{\bar{X}_F}$ is zero at $\bar{E}_{i,F}$, 
that is $r\psi\in (\bar{S}_F)_{\tau_i}-(\bar{S}_F)_{\tau_i}$. Therefore 
$d_{\tau_i\prec F} r\psi\in S_{\tau_i}-S_{\tau_i}$.

Let $d$ be a positive integer such that $X$ is $d$-orientable, and 
$d_{\tau_i\prec F}\mid d$ for all $\tau_i\prec F$. Then $X$ is $dr$-orientable and
$m=dr\psi$ satisfies the properties of Lemma~\ref{incrred}.c), hence $\omega^{[dr]}_{(X/k,B)}$ is invertible.

Suppose $(X/k,B)$ is a weakly normal log pair. It has wlc singularities if and only if 
each irreducible component of $(\bar{X}/k,\bar{C}+\bar{B})$ has lc singularities. This holds if and only if 
the coefficients of $B$ are at most $1$, or equivalently, $\psi\in F$ for every facet $F$.
\end{proof}

\begin{cor} $(X/k,B)$ has slc singularities if and only $X$ has at most nodal singularities in codimension one, 
and the components of the normalization $(\bar{X}/k,\bar{C}+\bar{B})$ are toric normal log pairs with lc singularities
having the same log discrepancy function.
\end{cor}

\begin{proof} See Example~\ref{2ex}.3) for the combinatorial criterion for $X$ to be at most nodal in codimension
one. In particular, $X$ is $2$-orientable. We may apply Proposition~\ref{wlpC}.
\end{proof}

\begin{lem} Suppose $(X/k,B)$ is a weakly normal log pair, with log discrepancy function $\psi$.
\begin{itemize} 
\item[1)] $(X/k,B)_{-\infty}=\cup_{b_i>1}E_i$ and $(\bar{X}/k,\bar{C}+\bar{B})_{-\infty}=
\sqcup_F \cup_{E_i\subset F,b_i>1}\bar{E}_{i,F}=\pi^{-1}((X,B)_{-\infty})$.
In particular, $\pi$ maps lc centers onto lc centers.
\item[2)] The lc centers of $(X/k,B)$ are $X_\sigma$, where $\psi\in \sigma\in \Delta$ and 
$\sigma\not\subset \tau_i$ if $b_i>1$.
\item[3)] Suppose $(X/k,B)$ is wlc.  Let $Z=X_\sigma$ be an lc center of $(X/k,B)$.
Then $\pi^{-1}(Z)$ is a disjoint union of lc centers, one for each irreducible component of $\bar{X}$:
$$
\pi^{-1}(Z)=\sqcup_F (\bar{X}_F)_{F\cap \sigma}.
$$
Some components of $\pi^{-1}(Z)$ may not dominate $Z$.
\end{itemize}
\end{lem}

\begin{proof} 1) We have $(\bar{X}/k,\bar{C}+\bar{B})_{-\infty}=
\sqcup_F \cup_{E_i\subset F,b_i>1}\bar{E}_{i,F}$. Its image $(X/k,B)_{-\infty}$ on $X$
equals $\cup_{b_i>1}E_i$. The inclusion $(\bar{X}/k,\bar{C}+\bar{B})_{-\infty}\subseteq
\pi^{-1}((X,B)_{-\infty})$ is clear, while the converse may be restated as follows: 
if $(\bar{X},\bar{C}+\bar{B})$ is lc at a closed point $\bar{x}$, then 
$(X,B)$ is wlc at $\pi(\bar{x})$. To prove this, we may localize and suppose 
$\pi(\bar{x})$ belongs to the closed orbit of $X$. If $F$ is the facet such that $\bar{x}\in \bar{X}_F$,
it follows that $\bar{x}$ belongs to the closed orbit of $\bar{X}_F$. We know that the toric log pair
$(\bar{X}/k,\bar{C}+\bar{B})|_{\bar{X}_F}$ has lc singularities at $\bar{x}$, a point belonging to its
closed orbit. Then $(\bar{X}/k,\bar{C}+\bar{B})|_{\bar{X}_F}$ has lc singularities. That is $\psi\in F$.

Let $F'$ be a facet of $\Delta$. Since $\Delta$ is $1$-connected, there exists a chain of facets
$F=F_0,F_1,\ldots,F_l=F'$ such that $F_i\cap F_{i+1}\ (0\le i<l)$ has codimension one.
We know $\psi\in F_0$. The codimension one face $\tau=F_0\cap F_1$ defines an irreducible 
component $X_\tau$ of $C$. Therefore $(\bar{X}_{F_0})_\tau$ appears as an irreducible component 
of $\bar{C}$. It is an lc center of $(\bar{X}/k,\bar{C}+\bar{B})|_{\bar{X}_{F_0}}$, that is $\psi\in \tau$.
Therefore $\psi\in F_1$. Repeating this argument along the chain, we obtain $\psi\in F'$.

We conclude that $\psi\in F$ for every facet $F$, that is $(\bar{X}/k,\bar{C}+\bar{B})$ has lc singularities.
Therefore $(X/k,B)$ has wlc singularities.

2) This follows from 1) and the description of the lc centers on the normalization.

3) This is clear.
\end{proof}

Suppose $(X/k,B)$ has wlc singularities.
The lc centers of $(X/k,B)$ are $X_\sigma$, where $\psi\in \sigma \in \Delta$.
Any union of lc centers is weakly normal.
The intersection of two lc centers is again an lc center. With respect to inclusion, there exists a unique 
minimal lc center, namely $X_{\sigma(\psi)}$ for $\sigma(\psi)=\cap_{\psi\in \sigma\in \Delta}\sigma$ 
(the unique cone of $\Delta$ which contains $\psi$ in its relative interior).

\begin{lem} 
Suppose $(X/k,B)$ is wlc. Then the minimal lc center of $(X/k,B)$ is normal.
\end{lem}

\begin{proof} Same as for Lemma~\ref{mlcn}.
\end{proof}

\begin{exmp}
Let $X=\Spec k[\cM]$ be weakly normal and $S_2$. Let $B\subset X$ be the reduced sum of 
invariant codimension one subvarieties at which $X/k$ is smooth (i.e. $B=\Sigma_X-C$). 
Then $(X/k,B)$ is a weakly normal log variety with wlc singularities if and only if $X$ is $\Q$-orientable.
Moreover, $\omega^{[2r]}_{(X/k,B)}\simeq \cO_X$ if and only if $X$ is $2r$-orientable.

Indeed, suppose $X$ is $2r$-orientable. The log discrepancy function $\psi$ is zero.
The forms $\{c_F\omega_F^{\otimes 2r}\}_F$ glue to a nowhere zero global section 
$\omega\in \Gamma(X,\omega^{[2r]}_{(X/k,B)})$, and the log crepant structure 
$(\bar{X},\bar{C}+\bar{B}=\Sigma_{\bar{X}})$ induced on the normalization has log canonical singularities.
\end{exmp}

The lc centers of $(X/k,B)$ of codimension one are the invariant primes $E_i$ such that
either $\mult_{E_i}B=1$, or $E_i$ is an irreducible component of $C$. The normalization of 
$E_i$ is $E^n_i=\Spec k[(S_{\tau_i}-S_{\tau_i})\cap \tau_i]$, the different $B_{E^n_i}$
is induced by the log discrepancy function $\psi$ of $(X/k,B)$, and the residue 
of $\{c_F\chi^{r\psi}\omega_F^{\otimes r}\}_F$ is $c_i \chi^{r\psi}\omega_i^{\otimes r}$. 


\subsection{The LCS locus}

Let $(X/k,B)$ be a toric weakly normal log pair, with wlc singularities. Let $\psi$ be its log discrepancy
function. The {\em LCS locus}, or {\em non-klt locus} of $(X/k,B)$, is the union $Y$ of all lc centers
of positive codimension in $X$. The zero codimension lc centers are exactly the irreducible 
components of $X$. Therefore $Y$ is the union of all $X_\sigma$ such that 
$\psi\in \sigma\in \Delta$, and $\sigma$ is strictly contained in some facet of $\Delta$.

\begin{prop}\label{lcs1}
$Y$ is weakly normal and $S_2$, of pure codimension one in $X$. Moreover, 
$Y$ is Cohen Macaulay if so is $X$.
\end{prop}

\begin{proof} Let $\pi\colon \bar{X}\to X$ be the normalization. Let $\bar{Y}=\pi^{-1}(Y)$. 
Then $\bar{Y}=\LCS(\bar{X}/k,\bar{C}+\bar{B})$. Since $Y$ contains $C$, the cartesian diagram 
\[ 
\xymatrix{
  \bar{X} \ar[d]_\pi & \bar{Y} \ar@{_{(}->}[l]  \ar[d]^\pi    \\
    X       & Y  \ar@{_{(}->}[l]  
} \]
is also a push-out. Equivalently, we have a Mayer-Vietoris short exact sequence 
$$
0\to \cO_X\to \pi_*\cO_{\bar{X}}\oplus \cO_Y\to \pi_*\cO_{\bar{Y}}\to 0.
$$
The subvariety $Y$ is weakly normal, since $X$ is. It is of pure codimension one in $X$,
since $Y=C\cup\Supp(B^{=1})$. We verify Serre's property in two steps.

{\em Step 1}: If $(X/k,B)$ is a {\em normal} toric log pair with lc singularities, then $Y=\LCS(X/k,B)$ is Cohen Macaulay. 

Indeed, let $X=\Spec k[M\cap \sigma]$ and $\psi\in \sigma$ be 
the log discrepancy function. Let $\tau\prec \sigma$ be the unique face which contains 
$\psi$ in its relative interior. In particular, a face of $\sigma$ contains $\psi$ if and only if 
it contains $\tau$. Then $Y=\cup_{\tau\prec \tau'\subsetneq \sigma} X_{\tau'}$.
Consider the quotient $M\to M'=M/(M\cap \tau-M\cap \tau)$, let $\sigma'$ be the image
of $\sigma$, denote $X'=\Spec k[M'\cap \sigma']$ and $T''=\Spec k[M\cap \tau-M\cap \tau]$. 
Then $X'$ is a normal affine variety with a fixed point $P$, and $Y\simeq T''\times \Sigma_{X'}$ 
(using the construction in~\cite[Remark 2.19]{TFR1}, we reduced to the case $\psi=0$). 
Since $T''$ is smooth, it is Cohen Macaulay.
By~\cite[Lemma 3.4.1]{Dan78}, $\depth_P(\Sigma_{X'})=\dim \Sigma_{X'}$, that is 
$X'$ is Cohen Macaulay. Therefore $Y$ is Cohen Macaulay.

{\em Step 2}: The disjoint union of normal affine toric varieties $\bar{X}$ is Cohen 
Macaulay by~\cite{Hoc72}, and $\bar{Y}$ is Cohen Macaulay by Step 1).
The Mayer-Vietoris short exact sequence and the cohomological interpretation 
of Serre's property, give that $Y$ is $S_2$ (respectively Cohen Macaulay) if so is $X$.
\end{proof}

Note that $\LCS(X/k,B)$ becomes the union of codimension one lc centers.
The normalizations of $\bar{Y}$ and $Y$ induce a commutative diagram
\[ 
\xymatrix{
  \bar{X} \ar[d]_\pi & \bar{Y} \ar@{_{(}->}[l]  \ar[d] &   \bar{Y}^n  \ar[l]  \ar[d]^g    \\
    X       & Y  \ar@{_{(}->}[l]  &   Y^n  \ar[l]_n
} \]

Let $X=\cup_F X_F$ and $Y=\cup_j E_j$ be the irreducible decompositions.
We have $\bar{X}=\sqcup_F \bar{X}_F$, $\bar{Y}=\sqcup_F \LCS(\bar{X}_F,(\bar{C}+\bar{B})|_{\bar{X}_F})$
and $\LCS(\bar{X}_F,(\bar{C}+\bar{B})|_{\bar{X}_F})=(\bar{C}\cup \Supp(\bar{B}^{=1}))|_{\bar{X}_F}$.
The irreducible components of $\bar{Y}$ are normal. Therefore
$
\bar{Y}^n=\sqcup_F\sqcup_{\psi\in \tau_j\prec F} \bar{E}_{j,F}=\sqcup_{\tau_j\ni \psi} \sqcup_{F\succ \tau_j}\bar{E}_{j,F}.
$
The normalization of $Y$ decomposes as $Y^n=\sqcup_j E^n_j$, with 
$E^n_j=\Spec k[(S_{\tau_j}-S_{\tau_j})\cap \tau_j]$.

Pick $r\ge 1$ such that such that $rB$ has integer coefficients and $\omega^{[r]}_{(X/k,B)}$ is invertible. 
Equivalently, $r\psi\in S_{\sigma(\Delta)}$ and there
exists a nowhere zero global section $\omega\in \Gamma(X,\omega^{[r]}_{(X/k,B)})$ such that
$\pi^*\omega|_{\bar{X}_F}=c_F\chi^{r\psi}\omega_F^{\otimes r}$ and 
$
\Res^{[r]}_{\bar{E}_{i,F}}\pi^*\omega=\pi^*_{i,F}(c_i\chi^{r\psi}\omega_i^{\otimes r}).
$
 Let $\eta$ be the rational 
pluridifferential form on $Y^n$ whose restriction to $E^n_j$ is $c_j\chi^{r\psi}\omega_j^{\otimes r}$. 
Then 
$$
\Res^{[r]}_{\bar{Y}^n} \pi^*\omega=g^*\eta.
$$
Let $B_{\bar{Y}^n}=-\frac{1}{r}(g^*\eta)$ and $B_{Y^n}=-\frac{1}{r}(\eta)$. Then $B_{\bar{Y}^n}$ is the
discriminant of $(\bar{X},\bar{C}+\bar{B})$ after codimension one adjunction to the components of $\bar{Y}^n$,
which is effective if $B$ is effective. Moreover, $g\colon (\bar{Y}^n, B_{\bar{Y}^n})\to (Y^n,B_{Y^n})$ is log
crepant. In particular $B_{Y^n}=g_*(B_{\bar{Y}^n})$ is effective if $B$ is effective.
All normal toric log pair structures induced on the irreducible components of $(\bar{X},\bar{C}+\bar{B})$,
$(\bar{Y}^n, B_{\bar{Y}^n})$ and $(Y^n,B_{Y^n})$ have the same log discrepancy function, namely $\psi$.
The correspondence $\omega\mapsto \eta$ induces the residue isomorphism
$$
\Res_{X\to Y^n} ^{[r]}\colon \omega^{[r]}_{(X/k,B)}|_{Y^n} \isoto \omega^{[r]}_{(Y^n/k,B_{Y^n})}
$$

\begin{prop}\label{lcs2} 
Let $r\in 2\Z$ such that $rB$ has integer coefficients and $\omega^{[r]}_{(X/k,B)}$ 
is invertible. The following are equivalent:
\begin{itemize}
\item[1)] There exists an invariant boundary $B_Y$ on $Y$ such that 
$(Y/k,B_Y)$ becomes a weakly normal log pair with the same log discrepancy function $\psi$,
with induced log structure $(Y^n,B_{Y^n})$ on the normalization, and such that 
codimension one residues onto the components of $Y^n$ glue to a (residue) isomorphism
$$
\Res_{X\to Y} ^{[r]}\colon \omega^{[r]}_{(X/k,B)}|_Y\isoto \omega^{[r]}_{(Y/k,B_Y)}.
$$
Moreover, $rB_Y$ has integer coefficients, and $B_Y$ is effective if so is $B$.
\item[2)] $(d_{Q\subset E_1}d_{E_1\subset X_F})^r=(d_{Q\subset E_2}d_{E_2\subset X_F})^r$ in $k^\times$, 
if $Q$ is an irreducible component of the non-normal locus of $Y$, $X_F$ is an irreducible  
component of $X$ containing $Q$, and $E_1,E_2$ are the (only) codimension one invariant 
subvarieties of $X_F$ containing $Q$.
\end{itemize}
\end{prop}

\begin{proof} If $Y$ is normal, there is nothing to prove. Suppose $Y$ is not normal.
Let $Q$ be an irreducible component of the non-normal locus of $Y$.
Then $Q=X_\gamma$ for some cone $\gamma\in \Delta$ of codimension two.
The primes of $Y^n$ over $Q$ are $Q_{\gamma,j}=\Spec k[(S_{\tau_j}-S_{\tau_j})\cap \gamma]\subset E^n_j$, 
one for each $\tau_j$ which contains $\gamma$. The induced morphism $n_{\gamma,j}\colon Q_{\gamma,j} \to Q$ is 
finite surjective, of degree $d_{Q\subset E_j}$.
Let $\omega_Q$ be a volume form on the torus inside $Q$, induced by an orientation of $S_\gamma-S_\gamma$.

We have $Q=E_1\cap E_2$ for some irreducible components $E_1,E_2$ of $Y$ (by the argument 
of the proof of Corollary~\ref{ncn}). Since $E_1,E_2$ are lc centers of $(X/k,B)$, so is their intersection
$Q$. That is $\psi\in \gamma$. Therefore $\mult_{Q_{j,\gamma}} B_{Y^n}=1$ for every $E_j\supset Q$. 
Since $r$ is even, we compute
$$
\Res^{[r]}_{Q_{j,\gamma}}\eta=c_j\chi^{r\psi}(\Res_{Q_{j,\gamma}}\omega_j)^{\otimes r}=
n_{j,\gamma}^*(c_j d^{-r}_{Q\subset E_j}\omega_Q^{\otimes r}).
$$

Property 1) holds if and only if $\Res^{[r]}_{Q_{j,\gamma}}\eta$ does not depend on $j$,
that is $c_i d^{-r}_{Q\subset E_i}=c_Q$ for every $E_i \supset Q$ (it follows that $E_i$
is an lc center, hence an irreducible component of $Y$). Since $c_i=c_Fd_{E_i\subset X_F}^{-r}$,
property 1) holds if and only if $c_F (d_{Q\subset E_i}d_{E_i\subset X_F})^{-r}=c_Q$ for every $Q\subset E_i\subset X_F$.

$1)\Longrightarrow 2)$: $c_F (d_{Q\subset E_1}d_{E_1\subset X_F})^{-r}=c_Q=c_F (d_{Q\subset E_2}d_{E_2\subset X_F})^{-r}$.
Therefore 2) holds.

$2)\Longrightarrow 1)$: We claim that $c_F (d_{Q\subset E_i}d_{E_i\subset X_F})^{-r}$ depends only on $Q$.
By 2), it does not depend on the choice of $E_i$, once $F$ is chosen. It remains to verify independence on $F$ as well. Since $\Delta$
is $1$-connected, we may only consider two facets $F,F'$ which contain $\gamma$, and intersect in codimension one.
Let $\tau_i=F\cap F'$. From $c_Fd_{E_i\subset X_F}^{-r}=c_i=c_{F'}d_{E_i\subset X_{F'}}^{-r}$, we obtain
$c_F(d_{Q\subset E_i}d_{E_i\subset X_F})^{-r}= c_F(d_{Q\subset E_i}d_{E_i\subset X_{F'}})^{-r}$. Therefore
$c_F (d_{Q\subset E_i}d_{E_i\subset X_F})^{-r}$ does not depend on $F$ either, say equal to $c_Q$. We obtain 
$$
\Res^{[r]}_{Q_{j,\gamma}}\eta=n_{j,\gamma}^*(c_Q \omega_Q^{\otimes r}).
$$
Therefore $(Y/k,B_Y=n_*(B_{Y^n}-\text{Cond}(n)))$ is a weakly normal log pair, $rB_Y$ has integer coefficients
and $\omega^{[r]}_{(Y/k,B_Y)}$ is trivialized by a nowhere zero global section such that 
$n^*\omega'=\eta$. The map $\omega\mapsto \omega'$ induces an isomorphism
$
\Res_{X\to Y}^{[r]}\colon \omega^{[r]}_{(X/k,B)}|_Y\isoto \omega^{[r]}_{(Y/k,B_Y)}.
$
\end{proof}


\section{Residues to lc centers of higher codimension}


\begin{defn}
We say that $X=\Spec k[\cM]$ has {\em normal components} if each irreducible component $X_F$ of $X$ is normal. 
\end{defn}
 
Suppose $X$ has normal components. If $F$ is a facet of $\Delta$ and $\sigma\prec F$,
then $S_\sigma=(S_F-S_F)\cap \sigma$. Therefore each invariant closed irreducible subvariety $X_\sigma\ (\sigma\in \Delta)$ 
is normal. Moreover, $X/k$ is weakly normal, and it is $S_2$ if and only if $\Delta$ is $1$-connected.
 
For the rest of this section, let $(X/k,B)$ be a toric weakly normal log pair with wlc singularities, such that 
{\em $X$ has normal components}. Under the latter assumption (which implies that $X$ is
$2$-orientable), $(X/k,B)$ is a wlc log pair if and only if the toric log structures induced on the irreducible
components of the normalization of $X$ have the same log discrepancy function $\psi\in \cap_F F$.
Let $r\in 2\Z$. Suppose $r\psi\in \cap_F S_F$, that is $rB$ has integer coefficients and 
$\omega^{[r]}_{(X/k,B)}$ is invertible. The lc centers of $(X/k,B)$ are $\{X_\sigma;\psi\in \sigma\in \Delta\}$. 
Let $X_\sigma$ be an lc center.
Let $B_{X_\sigma}$ be the invariant boundary induced by $\psi\in \sigma$. 
Then $(X_\sigma/k,B_{X_{\sigma}})$ becomes a {\em normal} toric log pair with lc singularities, 
$rB_{X_{\sigma}}$ has integer coefficients (effective if so is $B$) and $\omega^{[r]}_{(X_\sigma/k,B_{X_\sigma})}$ is 
trivial, and the lc centers of $(X_\sigma/k,B_{X_{\sigma}})$ are exactly the lc centers of $(X/k,B)$ which are
contained in $X_\sigma$. Let $\omega_\sigma$ be a volume form on the torus inside
$X_\sigma$ induced by some orientation of the lattice $S_\sigma-S_\sigma$.
The forms $\{\chi^{r\psi}\omega_F^{\otimes r} \}_F$ glue to a nowhere 
zero global section of $\omega^{[r]}_{(X/k,B)}$.

Let $Z$ be an lc center of $(X/k,B)$. On an irreducible toric variety, any proper 
invariant closed irreducible subvariety is contained in some invariant codimension one subvariety. 
Therefore we can construct a chain of invariant closed irreducible subvarieties
$$
X\supset X_0\supset X_1\supset \cdots \supset X_{c-1}\supset X_c=Z
$$
such that $X_0$ is an irreducible component of $X$ and $\codim(X_j \subset X_{j-1})=1 \ (0< j\le c)$. 
Let $X_i=X_{\sigma_i}$. Since $\sigma_c$ contains $\psi$, each $\sigma_i$ contains $\psi$. Therefore
each $X_i$ is an lc center of $(X/k,B)$, and $X_j$ becomes a codimension one lc center of $(X_{j-1}/k,B_{X_{j-1}})$.
Define the codimension zero residue $\Res^{[r]}_{X\to X_0} \colon \omega^{[r]}_{(X/k,B)}|_{X_0}\isoto  \omega^{[r]}_{(X_0/k,B_{X_0})}$ 
as the pullback to the normalization of $X$, followed by the restriction to the irreducible component $X_0$ of $\bar{X}$. We have
$$
\Res^{[r]}_{X\to X_0} \{\chi^{r\psi}\omega_F^{\otimes r} \}_F= \chi^{r\psi}\omega_{\sigma_0}^{\otimes r}.
$$
For $0<j\le c$, let $\Res^{[r]}_{X_{j-1}\to X_j} \colon \omega^{[r]}_{(X_{j-1}/k,B_{X_{j-1}})}|_{X_j}\isoto  \omega^{[r]}_{(X_j/k,B_{X_j})}$ 
be the usual codimension one residue. We have $\Res_{X_{j-1}\to X_j} \omega_{\sigma_j} = \epsilon_{j-1,j}\omega_{\sigma_j}$ for some
$\epsilon_{j-1,j} = \pm 1$. Since $r$ is even, we obtain 
$$
\Res^{[r]}_{X_{j-1}\to X_j}  \chi^{r\psi}\omega_{\sigma_{j-1}}^{\otimes r}= \chi^{r\psi}\omega_{\sigma_j}^{\otimes r}.
$$
The composition $\Res^{[r]}_{X_{c-1} \to X_c}|_Z \circ \cdots \circ \Res^{[r]}_{X_0 \to X_1}|_Z \circ \Res^{[r]}_{X \to X_0}|_Z$ is
an isomorphism $\omega^{[r]}_{(X/k,B)}|_Z\isoto \omega^{[r]}_{(Z/k,B_Z)}$ which maps 
$\{\chi^{r\psi}\omega_F^{\otimes r} \}_F$ onto $\chi^{r\psi} \omega_{\sigma_c}^{\otimes r}$. It does not depend on the 
choice of the chain from $X$ to $Z$, so we can denote it 
$$
\Res^{[r]}_{X\to Z}\colon \omega^{[r]}_{(X/k,B)}|_Z\isoto \omega^{[r]}_{(Z/k,B_Z)},
$$
and call it the {\em residue from $(X/k,B)$ to the lc center $Z$}.

\begin{lem}
Let $Z'$ be an lc center of $(Z/k,B_Z)$. Then $Z'$ is also an lc center of $(X/k,B)$, and the following diagram is commutative:
\[ 
\xymatrix{
  \omega^{[r]}_{(X/k,B)}|_{Z'} \ar[rr]^{\Res^{[r]}_{X\to Z'} }  \ar[dr]_{ (\Res^{[r]}_{X\to Z})|_{Z'}  }  &   & \omega^{[r]}_{(Z'/k,B_{Z'})} \\
  &  \omega^{[r]}_{(Z/k,B_Z)}|_{Z'} \ar[ur]_{ \Res^{[r]}_{Z\to Z'} }  &
} \]
\end{lem}

\begin{proof} Let $Z=X_\sigma$ and $Z'=X_{\sigma'}$. Then $\sigma'\prec \sigma$, and the generators are mapped as follows 
\[ 
\xymatrix{
  \{ \chi^{r\psi}\omega_F^{\otimes r} \}_F \ar[rr]  \ar[dr] &   &  \chi^{r\psi}\omega_{\sigma'}^{\otimes r}   \\
  &  \chi^{r\psi}\omega_\sigma^{\otimes r}  \ar[ur] &
} \]
Therefore the triangle of isomorphisms commutes.
\end{proof}

We may define residues onto lc centers in a more invariant fashion.

\begin{prop}\label{ita} Suppose $Y=\LCS(X/k,B)$ is non-empty. Then $(Y/k,B_Y)$ is a toric weakly normal log pair
with wlc singularities, such that $Y$ has normal components, and the codimension one residues onto the 
components of $Y$ glue to a residue isomorphism
$$
\Res^{[r]}_{X\to Y}\colon \omega^{[r]}_{(X/k,B)}|_Y \isoto \omega^{[r]}_{(Y/k,B_Y)}.
$$
Moreover, the lc centers of $(Y/k,B_Y)$ are exactly the lc centers of $(X/k,B)$ which are not maximal 
with respect to inclusion.
\end{prop}

\begin{proof} Since $X$ has normal components, so does $Y$. In particular, $Y/k$ is weakly normal. 
It is $S_2$ by Proposition~\ref{lcs1}.
Since $X$ has normal components, the incidence numbers $d_{E_i\subset X_F}$ are all $1$.
Therefore the condition 2) of Proposition~\ref{lcs2} holds, and the codimension one residues glue to
a residue onto $Y$.
\end{proof}

Iteration of the restriction to $\LCS$-locus induces a chain 
$
X=X_0\supset X_1\supset \cdots\supset X_c=W
$
with the following properties:
\begin{itemize}
\item $(X_0/k,B_{X_0})=(X/k,B)$.
\item $X_i=\LCS(X_{i-1}/k,B_{X_{i-1}})$ and $B_{X_i}$ is the different of $(X_{i-1}/k,B_{X_{i-1}})$ on $X_i$.
\item $\LCS(W/k,B_W)=\emptyset$. That is $W/k$ is normal and the coefficients of $B_W$ are strictly less than $1$.
\end{itemize}

The irreducible components of $X_i$ are the lc centers of $(X/k,B)$ of codimension $i$,
and $W$ is the (unique) minimal lc center of $(X/k,B)$. We compute
$$
\Res^{[r]}_{X \to W} = \Res^{[r]}_{X_{c-1} \to X_c}|_W \circ \cdots \circ \Res^{[r]}_{X_0 \to X_1}|_W.
$$ 
If $Z$ is an lc center of $(X/k,B)$ of codimension $i$, then $Z$ is an irreducible component of $X_i$, and 
$$
\Res^{[r]}_{X \to Z}= \Res^{[r]}_{X_i \to Z} \circ \Res^{[r]}_{X_{i-1} \to X_i}|_Z \circ \cdots \circ \Res^{[r]}_{X_0 \to X_1}|_Z,
$$
where $\Res^{[r]}_{X_i \to Z}$ is defined as the pullback to the normalization of $X_i$, followed by the
restriction to the irreducible component $Z$.

\begin{lem}
Let $X'$ be a union of lc centers of $(X/k,B)$, such that $X'$ is $S_2$. Then $(X',B_{X'})$
is a toric log pair with wlc singularities and the same log discrepancy function $\psi$, and 
residues onto the components of $X'$ glue to a residue isomorphism
$$
\Res^{[r]}_{X \to X'} \colon \omega^{[r]}_{(X/k,B)}|_{X'} \isoto \omega^{[r]}_{(X'/k,B_{X'})}.
$$
\end{lem}

\begin{proof}
Note that $X'$ has normal components, hence it is weakly normal. Since $X'$ is $S_2$,
all irreducible components have the same codimension, say $i$, in $X$. Then 
$X'$ is a union of some irreducible components of $X_i$. Define 
$$
\Res^{[r]}_{X \to X'}= \Res^{[r]}_{X_i \to X'} \circ \Res^{[r]}_{X_{i-1} \to X_i}|_{X'} \circ \cdots \circ \Res^{[r]}_{X_0 \to X_1}|_{X'}.
$$
The codimension zero residue $\Res^{[r]}_{X_i \to X'}$ is defined as the pullback to the normalization of $X_i$,
followed by restriction to the union of irreducible components consisting of the normalization of $X'$, followed 
by descent to $X'$.
\end{proof}

\begin{exmp}\label{tb}
Let $X=\Spec k[\cM]$ be $S_2$, with normal components. Let $B=\Sigma_X-C_X$, the reduced sum 
of invariant prime divisors at which $X/k$ is smooth. Then $(X/k,B)$ is a toric weakly normal log variety, 
with log discrepancy function $\psi=0$, and $\LCS(X/k,B)=\Sigma_X$.

Indeed, $X$ is $2$-orientable since it has normal components. The $2$-forms $\{\omega_F^{\otimes 2}\}_F$
glue to a nowhere zero global section of $\omega^{[2]}_{(X/k,B)}$. Since $\psi=0$, the lc centers are
the invariant closed irreducible subvarieties of $X$. Therefore $\LCS(X/k,B)=\Sigma_X$.
\end{exmp}

\begin{prop}\label{mex}
Let $X=\Spec k[\cM]$ be $S_2$, with normal components. Let $X_i$ be the union of codimension $i$ invariant
subvarieties of $X$. Then $X_i$ is $S_2$ with normal components, $X_{i+1}\subset X_i$ has pure codimension one
if non-empty, and coincides with the non-normal locus of $X_i$ if $i>0$, and the following properties hold:
\begin{itemize}
\item $(X/k,\Sigma_X-C)$ is a wlc log variety, with zero log discrepancy function, and $\LCS$-locus $X_1$. The
induced boundary on $X_1$ is zero, and we have a residue isomorphism
$$
\Res^{[2]}\colon \omega^{[2]}_{(X/k,\Sigma_X-C)}|_{X_1} \isoto \omega^{[2]}_{(X_1/k,0)}.
$$
\item For $i>0$, $(X_i/k,0)$ is a wlc log variety, with zero log discrepancy function, and $\LCS$-locus $X_{i+1}$.
The induced boundary on $X_{i+1}$ is zero, and we have a residue isomorphism
$$
\Res^{[2]}\colon \omega^{[2]}_{(X_i/k,0)}|_{X_{i+1}} \isoto \omega^{[2]}_{(X_{i+1}/k,0)}.
$$
\end{itemize}
\end{prop}

\begin{proof} By iterating the construction of Example~\ref{tb} and Proposition~\ref{ita}, 
we obtain for all $i\ge 0$ that $(X_i/k,B_{X_i})$ is a wlc log variety, with zero log discrepancy 
function, and $\LCS$-locus $X_{i+1}$, and the boundary induced on $X_{i+1}$ by 
codimension one residues is $B_{X_{i+1}}$.

If $X_i$ is a torus (i.e. $X$ contains no invariant prime divisors), then $X_{i+1}=\emptyset$. 
If $X_i$ is not a torus, then $X_{i+1}$ has pure codimension one in $X_i$.

Let $i>0$. We claim that $B_{X_i}=0$ and $X_{i+1}$ is the non-normal locus of $X_i$.
Suppose $X_i$ contains an invariant prime divisor $Q$. Since $i>0$, there exists an 
irreducible component $Q'$ of $X_{i-1}$ which contains $Q$. Then $Q$ has codimension 
two in $Q'$. Therefore $Q'$ has exactly two invariant prime divisors which contain $Q$, 
say $Q_1,Q_2$. Then $Q_1\ne Q_2$ are irreducible components of $X_i$, and $Q=Q_1\cap Q_2$. 
Therefore $Q$ is contained in $C_{X_i}$, the non-normal locus. We deduce $C_{X_i}=\Sigma_{X_i}=X_{i+1}$.
In particular, $B_{X_i}=0$.
\end{proof}


\subsection{Higher codimension residues for normal crossings pairs}


Let $(X/k,B)$ be a wlc log pair, let $x\in X$ be a closed point. We say that $(X/k,B)$ is 
{\em n-wlc at $x$}
if there exists an affine toric variety $X'=\Spec k[\cM]$ {\em with normal components}, 
associated to some monoidal complex $\cM$, an invariant boundary $B'$ on $X'$ and a closed point $x'$ in the 
closed orbit of $X'$, together with an isomorphism of complete local $k$-algebras 
$\cO_{X,x}^\wedge \simeq \cO_{X',x'}^\wedge$, and such that $(\omega^{[r]}_{(X/k,B)})_x^\wedge$
corresponds to $(\omega^{[r]}_{(X'/k,B')})_{x'}^\wedge$ for $r$ sufficiently divisible.
By~\cite{Ar69}, this is equivalent to the existence of a common \'etale neighborhood 
$$
\xymatrix{
   & (U,y)  \ar[dl]_i \ar[dr]^{i'} &  \\
(X,x)           &  & (X',x')  
}
$$ 
and a wlc pair structure $(U,B_U)$ on $U$ such that $i^* \omega^{[n]}_{(X/k,B)}=
\omega^{[n]}_{(U/k,B_U)}={i'}^* \omega^{[n]}_{(X'/k,B')}$ for all $n\in \Z$.
It follows that $X'/k$ must be weakly normal and $S_2$, and $(X'/k,B')$ is wlc.

Being n-wlc at a closed point is an open property. We say that {\em $(X/k,B)$ is n-wlc} if it so at every closed point. 
For the rest of this section, let $(X/k,B)$ be n-wlc. Let $r\in 2\Z$ such that $rB$ has integer coefficients and 
$\omega^{[r]}_{(X/k,B)}$ is invertible.

\begin{prop} Suppose $Y=\LCS(X/k,B)$ is non-empty. Then $Y$ is weakly normal and $S_2$, of pure
codimension one in $X$. There exists a unique boundary $B_Y$ such that $(Y/k,B_Y)$ is n-wlc, 
and codimension one residues onto the irreducible components of the normalization of $Y$ glue 
to a residue isomorphism
$$
\Res^{[r]}_{X\to Y}\colon \omega^{[r]}_{(X/k,B)}|_Y \isoto \omega^{[r]}_{(Y/k,B_Y)}.
$$
Moreover, the lc centers of $(Y/k,B_Y)$ are exactly the lc centers of $(X/k,B)$ which are not maximal 
with respect to inclusion.
\end{prop}

\begin{proof}
By Proposition~\ref{ita} for a local analytic model.
\end{proof}
 
Iteration of the restriction to $\LCS$-locus induces a chain 
$
X=X_0\supset X_1\supset \cdots\supset X_c=W
$
with the following properties:
\begin{itemize}
\item $(X_0/k,B_{X_0})=(X/k,B)$.
\item $(X_i/k,B_{X_i})$ is a n-wlc pair, $X_i=\LCS(X_{i-1}/k,B_{X_{i-1}})$ and $B_{X_i}$ is the different on $X_i$
of $(X_{i-1}/k,B_{X_{i-1}})$.
\item $\LCS(W/k,B_W)=\emptyset$. That is $W/k$ is normal and the coefficients of $B_W$ are strictly less than $1$.
\end{itemize}

The irreducible components of $X_i$ are the lc centers of $(X/k,B)$ of codimension $i$,
and $W$ is the union of lc centers of $(X/k,B)$ of largest codimension.

Let $Z$ be an lc center of $(X/k,B)$, of codimension $i$. 
Then $Z$ is an irreducible component of $X_i$. Let $Z^n\to Z$ be the normalization.
Then $Z^n$ is an irreducible component of the normalization of $X_i$. Let $B_{Z^n}$
be the induced boundary. Define the zero codimension residue 
$$
\Res^{[r]}_{X_i \to Z^n}\colon \omega^{[r]}_{(X_i/k,B_{X_i})}|_{Z^n}\isoto \omega^{[r]}_{(Z^n/k,B_{Z^n})}
$$ 
as the pullback from $X_i$ to its normalization, followed by the restriction to the irreducible component $Z^n$.
Define 
$
\Res^{[r]}_{X \to Z^n}= \Res^{[r]}_{X_i \to Z^n} \circ \Res^{[r]}_{X_{i-1} \to X_i}|_{Z^n} \circ \cdots \circ \Res^{[r]}_{X_0 \to X_1}|_{Z^n}.
$
We obtain:

\begin{thm}\label{nwa}
Let $(X/k,B)$ be n-wlc. Let $r\in 2\Z$ such that $rB$ has integer coefficients and $\omega^{[r]}_{(X/k,B)}$ is invertible.
Let $Z$ be an lc center, with normalization $Z^n\to Z$. Then there
exists a log pair structure $(Z^n,B_{Z^n})$ on $Z^n$, and a higher codimension residue isomorphism
$$
\Res^{[r]}_{X \to Z^n}\colon  \omega^{[r]}_{(X/k,B)}|_{Z^n}\isoto \omega^{[r]}_{(Z^n/k,B_{Z^n})}.
$$
Moreover, $B_{Z^n}$ is effective if so is $B$, and $rB_{Z^n}$ has integer coefficients.
\end{thm}

\begin{defn}
A {\em normal crossings pair} $(X/k,B)$ is an n-wlc pair with local analytic models of the following special type:
$0\in (X'/k,B')$, where $X'=\cup_{i\in I} H_i \subset \bA^n_k$ for some $I\subseteq \{1,\ldots,n\}$, and
$H_i=\{z_i=0\}\subset \bA_k^n$ is the $i$-th standard hyperplane. It follows that $B'=\sum_{i\notin I}b_iH_i|_{X'}$
for some $b_i\in \Q_{\le 1}$. 
\end{defn}

\begin{cor}\label{nwaNC}
Let $(X/k,B)$ be normal crossings pair. Let $r\in 2\Z$ such that $rB$ has integer coefficients and $\omega^{[r]}_{(X/k,B)}$ is invertible.
Let $Z$ be an lc center, with normalization $Z^n\to Z$. Then there exists a log pair structure $(Z^n,B_{Z^n})$ on $Z^n$, with log
smooth support, and a higher codimension residue isomorphism
$$
\Res^{[r]}_{X \to Z^n}\colon  \omega^{[r]}_{(X/k,B)}|_{Z^n}\isoto \omega^{[r]}_{(Z^n/k,B_{Z^n})}.
$$
Moreover, $B_{Z^n}$ is effective if so is $B$, and $rB_{Z^n}$ has integer coefficients.
\end{cor}

\begin{exmp}
Let $(X/\C,\Sigma)$ be a log smooth pair, that is $X/\C$ is smooth and $\Sigma$ is a divisor with normal
crossings in $X$. Let $Z$ be an lc center of $(X/\C,\Sigma)$, let $Z^n\to Z$ be the normalization. 
Deligne~\cite{Del71} defines a residue isomorphism
$
\Res \colon \omega_X(\log \Sigma)|_{Z^n}\isoto \omega_{Z^n}(\log \Sigma_{Z^n})\otimes \epsilon_{Z^n},
$
where $\epsilon_{Z^n}$ is a local system (orientations of the local analytic branches of $\Sigma$ through $Z$) 
such that $\epsilon_{Z^n}^{\otimes 2}\simeq \cO_{Z^n}$.
Then $\Res^{\otimes 2}$ coincides with 
$
\Res^{[2]}\colon  \omega^{[r]}_{(X/\C,\Sigma)}|_{Z^n}\isoto \omega^{[r]}_{(Z^n/k,\Sigma_{Z^n})}
$
defined above.
\end{exmp}

\begin{question}
Let $(X/k,B)$ be a wlc log pair which is locally analytically isomorphic to a toric wlc log pair
(the toric local model may have non-normal irreducible components).
Let $Z$ be an lc center, let $Z^n\to Z$ be the normalization. Is there a residue isomorphism from $X$ to $Z^n$? 
Is it torsion the moduli part in the higher codimension adjunction formula from $(X/k,B)$ to $Z^n$?
\end{question}


\end{document}